\documentclass{article}

\usepackage[T1]{fontenc}
\usepackage[utf8]{inputenc}
\usepackage[british]{babel}

\usepackage{lmodern}

\usepackage{amsmath}
\usepackage{amsthm}
\usepackage{amsfonts}
\usepackage{hyperref}
\hypersetup{
	colorlinks = true,
	linkcolor = purple, 
	filecolor = purple, 
	citecolor = purple,
	urlcolor = purple
}
\usepackage{url} 
\usepackage{amssymb}
\usepackage{tikz-cd}
\usetikzlibrary{babel}
\allowdisplaybreaks

\usepackage[backend=biber, style=alphabetic, sorting=nyt, maxcitenames=7,
mincitenames=1, maxnames=7, abbreviate=true, block=space]{biblatex}

\newtheorem{lemma}{Lemma}[section]
\newtheorem{theorem}[lemma]{Theorem}
\newtheorem{proposition}[lemma]{Proposition}
\newtheorem{corollary}[lemma]{Corollary}

\newtheorem*{theorem*}{Theorem}
\newtheorem*{proposition*}{Proposition}
\newtheorem*{corollary*}{Corollary}

\theoremstyle{definition}
\newtheorem{rem}[lemma]{Remark}
\newtheorem{example}[lemma]{Example}
\newtheorem{notation}[lemma]{Notation}
\newtheorem{definition}[lemma]{Definition}

\newcommand{\OO}{\mathcal{O}}
\newcommand{\A}{\mathbb{A}}
\newcommand{\PP}{\mathbb{P}}

\newcommand{\Z}{\mathbb{Z}}
\newcommand{\Q}{\mathbb{Q}}
\newcommand{\R}{\mathbb{R}}
\newcommand{\C}{\mathbb{C}}
\newcommand{\m}{\mathfrak{m}}

\newcommand{\I}{\mathcal{I}}

\newcommand{\critical}{\mathfrak{c}}
\newcommand{\K}{\mathcal{K}}
\newcommand{\sGW}{\mathcal{GW}}
\newcommand{\F}{\mathcal{F}}

\DeclareMathOperator{\Spec}{Spec}
\DeclareMathOperator{\Proj}{Proj}

\DeclareMathOperator{\Gr}{Gr}
\DeclareMathOperator{\id}{id}

\DeclareMathOperator{\rank}{rank}

\DeclareMathOperator{\Sym}{Sym}
\DeclareMathOperator{\pr}{pr}

\DeclareMathOperator{\End}{End}
\DeclareMathOperator{\GW}{GW}

\DeclareMathOperator{\Tr}{Tr}

\DeclareMathOperator{\sgn}{sgn}
\DeclareMathOperator{\Bl}{Bl}

\let\det\undefined
\DeclareMathOperator{\det}{det}

\DeclareMathOperator{\Aut}{Aut}

\newcommand{\SH}{\text{\textnormal{SH}}}
\newcommand{\Sm}{\text{\textnormal{Sm}}}

\newcommand{\MW}{\text{\textnormal{MW}}}

\newcommand{\topo}{\text{\textnormal{top}}}

\makeatletter
\newcommand{\colim@}[2]{%
  \vtop{\m@th\ialign{##\cr
      \hfil$#1\operator@font colim$\hfil\cr
      \noalign{\nointerlineskip\kern1.5\ex@}#2\cr
      \noalign{\nointerlineskip\kern-\ex@}\cr}}%
}
\newcommand{\fcolim}{%
  \mathop{\mathpalette\colim@{\rightarrowfill@\textstyle}}\nmlimits@
}
\newcommand{\colim}{%
  \mathop{\mathpalette\colim@{}}\nmlimits@
}
\makeatother

\title{Quadratic Euler Characteristic of Geometrically Cyclic Branched Coverings}
\author{Louisa F. Bröring}
\date{}

\begin{document}
\maketitle

\begin{abstract}
  For an $n$-fold geometrically cyclic branched covering $Y$ of a smooth, projective
  scheme $X$ branched at a smooth closed subscheme $Z\subset X$ with $n \in k^\times$, we compute the quadratic Euler
  characteristic of $Y$ in terms of certain Euler classes on $X$ and $Z$ using
  the quadratic Riemann-Hurwitz formula of Levine. When $n$ is
  odd, we relate the quadratic Euler characteristic of $Y$ to the quadratic
  Euler characteristics of $X$ and $Z$, obtaining similar formulae to the
  situation in topology.

  As an application, we compute the quadratic Euler characteristic of
  geometrically cyclic branched double coverings of $\PP^2$.
\end{abstract}

\tableofcontents

\section*{Introduction}
\addcontentsline{toc}{section}{Introduction}
\markboth{Introduction}{}

The quadratic Euler characteristic is a refinement of the Euler characteristic
in topology to algebraic geometry, first
considered by Hoyois \cite{HoyoisQRGLVTF}. Let $k$ be a perfect field of
characteristic not two. For $X$ a smooth, projective scheme over $k$, the
quadratic Euler characteristic $\chi(X/k)$ of $X$ is not an integer, but an
element of the Grothendieck-Witt ring $\GW(k)$ of (virtual, non-degenerate) quadratic
forms over $k$.

The quadratic Euler characteristic carries a lot of information: if $k$ is a
subfield of $\R$, the
rank of $\chi(X/k)$ equals the topological Euler characteristic of
$X(\C)$. Furthermore, the signature of $\chi(X/k)$ with respect to the
specified embedding into $\R$ equals the topological Euler characteristic of
$X(\R)$. A theorem of Saito allows to interpret the discriminant of
$\chi(X/k)$ in terms of the determinant of $\ell$-adic cohomology for any
$\ell$ that is coprime to the characteristic of the base-field. In general, it is hard to compute quadratic Euler characteristics.

Quadratic Euler characteristics are often used in the programme of refined
enumerative geometry, which aims at proving counts in enumerative geometry
over more general fields, often valued in $\GW(k)$. One can often deduce known invariants or results by
taking ranks or traces from the refined results.

An \emph{$n$-fold geometrically cyclic covering branched at $Z$}
is a map between smooth,
projective schemes $f\colon Y \to
X$ over $k$ together with a smooth, closed subscheme $Z\subset X$ satisfying
the following condition: there exists a line bundle $L \to X$
and a section $s \colon X \to L^{\otimes n}$ with zero subscheme $i \colon Z\hookrightarrow
X$ and a morphism $Y \to L$ such that we have a fibre square
\[
  \begin{tikzcd}
    Y \ar[r]\ar[d, "f"] & L\ar[d, "m_n"]\\
    X \ar[r, "s"] & L^{\otimes n}.
  \end{tikzcd}
\]
Here, $m_n\colon L \to L^{\otimes n}$ is the map sending a local generator $t$
of $L$ to $t^n$.

Geometrically cyclic branched coverings are an algebro-geometric analogue of
branched coverings in topology. For these coverings and $k = \C$,
one has the classical formula
\[
  \chi^\topo(Y(\C)) = n\chi^\topo(X(\C)) - (n-1)\chi^\topo(Z(\C)).
\]
This holds more generally for $k$ algebraically closed after replacing the
topological Euler characteristic by the étale Euler characteristic.
In this article, we investigate an analogue of this formula for the quadratic
Euler characteristic, that is, we compare the quadratic Euler
characteristic of $Y$ with the quadratic Euler characteristic of $X$ and
$Z$ when $X$ is equidimensional and $n$ is invertible in $k$.

Our main
result is as follows:

\begin{theorem*}[see Corollary~\ref{cor:odd-euler-char-general} and Corollary~\ref{cor:odd-euler-char-general-no-curve}]
  Let $X$ be equidimensional of dimension $r$ and let $n \ge 2$ be invertible in $k$.

  Then if $n$ is odd, we have
  \[
    \chi(Y/k) = (\langle n\rangle + \frac{n-1}{2}\cdot H)\cdot \chi(X/k)
     - \chi(Z/k)\cdot
                      \frac{n-1}{2}\cdot H
  \]
  in $\GW(k)$.

  If $n$ is even, assume that
  there is a smooth, projective curve $C$ over $k$ possessing a half-canonical line
  bundle together with a
  morphism $\rho\colon X \to C$ whose critical points $\critical(\rho)$ lie
  outside of $Z$. Furthermore, assume that $\rho$ and $\rho\circ i$ only have
  finitely many critical points. Choose for each $x \in \critical(\rho\circ i) \cup \critical(\rho)$ a local trivialisation of $L$ around
  $x$ and let $s_x \in \OO_{X,x}$ be the element corresponding with
  the section $s$ under the chosen trivialisation. For each $x \in
  \critical(\rho\circ i)$, let $t_x$ be a
  normalised parameter of $C$ at $\rho(x)$ and write $d\rho(dt_x) =
  \alpha_x ds_x$ in $\Omega^1_{X,x}\otimes k(x)$ with $\alpha_x \in
  k(x)^\times$. Then we have
  \[
    \chi(Y/k) = (2\cdot \langle n\rangle + \frac{n-2}{2}H)\cdot
    \chi(X/k) + \beta_X - (\langle 1\rangle + \frac{n-2}{2}\cdot H)\cdot
    \chi(Z/k) - \beta_Z
  \]
  in $\GW(k)$ with
  \[
    \beta_X = (-1)^r\sum_{x \in \critical(\rho)}\Tr_{k(x)/k}(e_{x}(d\rho)\cdot(\langle
    ns_x\rangle - \langle n\rangle))\in \GW(k)
  \]
  and
  \[
    \beta_Z = (-1)^{r}\sum_{y \in
      \critical(\rho\circ i)}\Tr_{k(y)/k}
    (e_y(d(\rho\circ i)) \langle (n\alpha_x\rangle -
    \langle 1\rangle)) \in \GW(k).
  \]
\end{theorem*}

Here, for $a \in k^\times$, let $\langle a \rangle \in \GW(k)$ denote
the quadratic form $x \mapsto ax^2$, and let $H$ denote the hyperbolic form
$\langle 1\rangle + \langle -1\rangle \in \GW(k)$.

As detailed in Remark \ref{rem:recover-classical-formula}, this result recovers the classical formula for $k =
\C$ by taking the rank.

If $n$ is odd and every $x \in \critical(\rho)$ has exactly $n$ preimages under
$f$, we can also describe the relation between the Euler characteiristics as follows:

\begin{corollary*}[see Corollary \ref{cor:odd-euler-char-splitting}]
  Let $X$ be equidimensional and let $n \ge 3$ be invertible in $k$. Assume that
  there is a smooth, projective curve $C$ over $k$ possessing a half-canonical line
  bundle together with a
  morphism $\rho\colon X \to C$ whose critical points $\critical(\rho)$ lie
  outside of $Z$. Furthermore, assume that $\rho$ and $\rho\circ i$ only have
  finietly many critical points, that $n$ is odd, and that $f^{-1}(x) \subset Y$ consists of $n$
  closed points for all $x
  \in \critical(\rho)$. Then we have
  \[
    \chi(Y/k) = n\cdot  \chi(X/k) -
    \chi(Z/k)\cdot \frac{n-1}{2}\cdot H\in \GW(k).
  \]
\end{corollary*}

The existence of $\rho$ in the above theorems is a mild
assumption and can always be satisfied by blowing up a sufficiently nice smooth, codimension two
subscheme of $X$; see Appendix \ref{sec:lefschetz} for details.

The morphism $\rho$ induces a global section $d\rho$ of $\Omega^1_{X/k}\otimes
\rho^\ast\omega_{C/k}^{-1}$, whose Euler class can be related to the quadratic
Euler characteristic of $X$ using the quadratic Riemann Hurwitz formula of
Levine \cite[Theorem~11.2]{LevineA}. The zeroes of $d\rho$ correspond
with the critical points $\critical(\rho)$ of $\rho$, and the Euler class of
$\Omega^1_{X/k}\otimes\rho^\ast\omega_{C/k}^{-1}$ can be computed using local
Euler classes $e_x(d\rho)$ of $d\rho$ at every $x \in
\critical(\rho)$. Using the pushforward map $\pi_{X\ast}i_{x\ast}$, we can
map $e_x(d\rho)$ to the Grothendieck-Witt ring $\GW(k)$.

The analysis of the critical points of $\rho\circ f$ splits up into the
analysis of the points over the critical points of $\rho$, studied in Section \ref{sec:etale-classes}, and the critical points of
$\rho\circ i$, studied in Section \ref{sec:odd-contribution}.

As an application, we use the above results to compute the quadratic Euler
characteristic of geometrically cyclic branched double coverings of
$\PP^2$ that have a rational point outside of the branch locus in Section \ref{sec:application-p2}. That is
\[
  \begin{tikzcd}
    & X \ar[r]\ar[d, "\varphi"] & \OO(n)\ar[d, "m_2"]\\
    C \ar[r, "i"] & \PP^2 \ar[r, "s"] & \OO(2n)
  \end{tikzcd}
\]
with $X(k) \setminus C(k) \ne \emptyset$.
Here, we obtain the following result:

\begin{theorem*}[see Theorem \ref{thm:k3:euler-char}]
  Choose a closed point $y_0 \in X(k)\setminus C(k)$ and parametrise $\PP^2$ such that $f(y_0)
  = [0:0:1] \in \PP^2$. Let $F \in k[X_0,X_1,X_2]$ be the polynomial of degree
  $2n$ corresponding with $s$ under the chosen parametrisation. Denote the
  dehomogenisations of $F$ with respect to $X_0$ and $X_1$ by $f_0$ and $f_1$,
  respectively.

  For every $y$ in the vanishing locus $V(\partial F/\partial X_2)$ of $\frac{\partial F}{\partial X_2}$, pick $i \in \{0,1\}$ such that $y
  \in D(X_{i})$. Pick a local parameter $t_y \in \OO_{C,y}$ and write
  $\frac{\partial f_{i}}{\partial x_2} = u_yt_y^{m_y}$ for $m_y \ge 1$ and $u_y \in
  \OO_{C,y}^\times$.
  Assume that $m_y$ is
  invertible in $k$ for all $y \in V(\partial F/\partial X_2)$ and write
  $\frac{\partial^2 f_i}{\partial x_2^2} = m_y\alpha_yt_y^{m_y-1}$ for some
  $\alpha_y \in \OO_{C,y}^\times$.
  
  Define
  \[
    \beta = \sum_{\substack{y \in \critical(\rho\circ i)\\m_y\text{ odd}}}
    \Tr_{k(y)/k}(\langle -2\alpha_y\rangle) \in \GW(k).
  \]
  Then $\rank\beta$ is even and
  \[
    \chi(X/k) = 2\langle1\rangle + \beta + ((2n-1)(n-1) + 1 -\frac 12 \rank
    \beta)\cdot H \in \GW(k).
  \]
\end{theorem*}

With this, one can, for example, compute the quadratic Euler characteristic of the
surface $X$ over $k$ induced by the section $s = aX_0^{2n} + bX_1^{2n} +
cX_2^{2n}$ of $\OO(2n)$ with $a,b,c\in k^\times$ when $2n$ and $2n-1$ are invertible in $k$. Here,
one obtains
\[
  \chi(X/k) = \langle 2\rangle + \langle 2ab\rangle + ((2n-1)(n-1) + 1)H \in \GW(k).
\]
See Proposition \ref{prop:k3:fermat-example} for details. If $n = 3$, this is a K3 surface.

Except for the results in Section \ref{sec:application-p2}, the results
presented in this article also form a chapter in the author’s
PhD-thesis, which was submitted on 22 January 2025.

\subsection*{Notation and Conventions}
\addcontentsline{toc}{subsection}{Notation and Conventions}
Throughout, let $k$ be a perfect field that is not of characteristic two. The
assumption that $k$ is perfect is not strictly necessary and we do so to avoid
technical complications.

A curve, is a one-dimensional, geometrically connected, separated, finite-type
scheme over $k$. For a
point $x \in X$ on a scheme $X$, we write
$k(x)$ for the residue field at $x$.

We write $\A^1$ for the affine line $\A^1_k$ over $k$ and $\PP^n$ for the
$n$-dimensional projective space $\PP^n_k$ over $k$.
If $X$ is a smooth, projective scheme over $k$, we write $\Omega_{X/k}$ or $\Omega^1_{X/k}$ for
the sheaf of Kähler differentials of $X$ over $k$.

For a quasi-compact and quasi-separated scheme $X$, let $\SH(X)$ denote the stable
motivic homotopy category of $X$ as described by Hoyois \cite{HoyoisSOEMHT},
and let $1_X\in \SH(X)$ be the
unit object with respect to the smash product in $\SH(X)$.

\subsection*{Acknowledgements}
\addcontentsline{toc}{subsection}{Acknowledgements}
First and foremost, I would like to extend my deepest gratitude to my adviser,
Marc Levine, for suggesting this project to me and for all
of his invaluable support over the past years.
I am also extremely grateful to Anna Viergever for our many insightful
discussions.
I would like to extend my sincere thanks to Sabrina Pauli for many fascinating
discussions and in particular her insights into computing Euler classes.
I also want to thank Jesse Pajwani for our discussions and insights around the
results in Section \ref{sec:etale-classes} and, in particular, for explaining
Proposition \ref{prop:real-etale-part-no-product} to me.
Moreover, I very much appreciate, Dhyan Aranha, Tom Bachmann, Linda Carnevale, Jochen Heinloth, Clémentine
Lemarié{-}{-}Rieusset, Fabien Morel, Herman Rohrbach, Hind Souly, and Johannes Sprang for
interesting discussions.
I am also grateful to Stefan Schreieder and the algebraic geometry group at
the Leibniz Universität Hannover for offering me a place to work when I was
visiting my partner and for the welcoming atmosphere that I encountered there.

This work was funded by Deutsche Forschungsgemeinschaft (DFG, German Research
Foundation) - Research Training Group 2553 - project number
412744520 and Deutsche Forschungsgemeinschaft (DFG, German Research Foundation)
- Advances in quadratic enumerative geometry - project number
568776200.

\section{The Quadratic Euler Characteristic and Euler Classes}
\label{chap:euler-char}
The quadratic Euler
characteristic is an analogue in algebraic geometry of the Euler characteristic
in topology. An important tool for computing quadratic Euler characteristic uses
the theory of Euler classes, which in this context are an analogue of top Chern
classes in motivic homotopy theory. In the following, we will briefly introduce
both concepts based on \cite{LevineA}.

\subsection{The Quadratic Euler Characteristic}
\label{sec:qec}
The quadratic Euler characteristic is constructed unsing motivic homotopy
thoery and was first studied by Hoyois \cite{HoyoisQRGLVTF}.
Based on work of Morel-Voevodsky \cite{MorelVoevodsky99HTS}, Voevodsky \cite{Voevodsky98HT} constructs
a stable motivic
homotopy category $\SH(S)$ for a Noetherian base scheme $S$ as an analogue of
the stable homotopy category in topology. The category $\SH(S)$ is a symmetric monoidal
category and we shall denote its unit object by $1_S \in SH(S)$. See Hoyois
\cite{HoyoisSOEMHT} for an introduction to $\SH(S)$. There is an infinite suspension functor
$\Sigma^\infty\colon \Sm_S \to \SH(S)$ from the category of smooth schemes
over $S$ and a construction of a six functor formalism for $\SH$,
which can be found in \cite{HoyoisSOEMHT}.

The quadratic Euler characteristic is defined using the notion of a strong dual as defined
by Dold-Puppe \cite{dold80-categoricalEC}.

\begin{definition}[{\cite[Theorem 1.3]{dold80-categoricalEC}}]
  An object $X$ in a symmetric monoidal category $\mathcal C$ is strongly dualisable if there exists
  an object $X^\vee \in \mathcal C$ together with morphisms $\delta_X\colon
  1_{\mathcal C} \to  X\otimes X^\vee$
  and $\operatorname{ev}_X\colon X^\vee \otimes  X \to 1_{\mathcal C}$ such that
  the compositions
  \[
    \begin{tikzcd}
      X \ar[r, "\cong", phantom] & 1_{\mathcal C}\otimes X \ar[r, "\delta_X\otimes \id"] &
      X\otimes X^\vee \otimes X \ar[r, "\id\otimes
      \operatorname{ev}_X"] &  X\otimes
      1_{\mathcal C} \ar[r, "\cong", phantom] & X
    \end{tikzcd}
  \]
  and
  \[
    \begin{tikzcd}
      X^\vee \ar[r, "\cong", phantom] & X^\vee\otimes 1_{\mathcal C} \ar[r, "\id\otimes\delta_X"] &
      X^\vee \otimes X\otimes X^\vee \ar[r, "\operatorname{ev}_X\otimes\id"] &
      1_{\mathcal C}\otimes X^\vee \ar[r, "\cong", phantom] & X^\vee
    \end{tikzcd}
  \]
  are the identity.
\end{definition}

\begin{rem}
  For a strongly dualisable object $X$, the triple
  $(X^\vee, \delta_X, \operatorname{ev}_X)$ is unique
  up to unique isomorphism, see for example
  \cite[Theorem 1.3 (c)]{dold80-categoricalEC} for a proof. This triple
  $(X^\vee, \delta_X, \operatorname{ev}_X)$ is called the \emph{dual} of $X$.
\end{rem}

Given a smooth, projective scheme $X$ over $k$ with structure map $p_X\colon X\to \Spec(k)$,
the object $\Sigma^\infty X = p_{X\#}1_X\in \SH(k)$ is strongly dualisable by
for example \cite[Theorem~5.22]{HoyoisSOEMHT}.

For a strongly dualisable object $A$ in a symmetric monoidal category
$\mathcal{C}$, the \emph{categorical Euler characteristic}
$\chi^{\text{cat}}(A)\in \End_{\mathcal{C}}(1_{\mathcal{C}})$ is the
categorical trace of the identity as defined by Dold-Puppe
\cite[Definiton~4.1]{dold80-categoricalEC}. That is, $\chi^{\text{cat}}(A)$ is the
composition
\[
  \begin{tikzcd}
    1_{\mathcal{C}} \ar[r, "\delta_A"] & A\otimes A^\vee \ar[r, "\cong",
    phantom] & A^\vee \otimes A \ar[r, "\operatorname{ev}_A"] & 1_{\mathcal{C}},
  \end{tikzcd}
\]
where $(A^\vee, \delta_A, \operatorname{ev}_A)$ is the dual of $A$, and the
isomorphism $A\otimes A^\vee \cong A^\vee \otimes A$ is the switching
isomorphism from the definition of a symmetric monoidal
category.

In order to define the quadratic Euler characteristic, we first need to define
the Grothendieck-Witt ring over $k$, which is where the quadratic Euler
characteristic is valued.

\begin{definition}
  The \emph{Grothendieck-Witt ring $\GW(k)$} over $k$ is defined as
  the group completion of the monoid of isometry classes of
  non-degenerate quadratic forms over $k$ with respect to orthogonal
  direct sums.

  The tensor product of quadratic forms induces a ring structure on the group
  $\GW(k)$.
\end{definition}

\begin{rem}
  \label{rem:gw-rel}
  Over a field $k$ whose characteristic is not $2$, non-degenerate quadratic
  forms are
  the same as non-degenerate symmetric bilinear forms and we get the
  following presentation of $\GW(k)$: the
  quadratic forms $\langle a \rangle; x \mapsto ax^2$ for $x \in k^\times$
  generate $\GW(k)$ and these forms are subject to the following relations for
  $a, b\in k^\times$:
  \begin{enumerate}
  \item $\langle a\rangle \cdot \langle b\rangle = \langle a b \rangle$,
    \label{rem:gw-rel:mult}
  \item
    $\langle a\rangle + \langle b\rangle = \langle a+b\rangle + \langle
    ab(a+b)\rangle$, whenever $a+ b \in k^\times$,
    \label{rem:gw-rel:add}
  \item $\langle ab^2\rangle = \langle a\rangle$, and
    \label{rem:gw-rel:square}
  \item
    $\langle a\rangle + \langle -a\rangle = \langle 1\rangle + \langle
    -1\rangle =: H$.
    \label{rem:gw-rel:hyp}
  \end{enumerate}
  The quadratic form $H = \langle 1\rangle + \langle -1\rangle$ is called the
  \emph{hyperbolic form}.

  The first description of this presentation is due to Witt
  \cite[Section~1]{WittQF}, and in the form presented here, this is \cite[Lemma~3.9]{MorelATF}. Note that the relation (\ref{rem:gw-rel:mult}) is not included in
  \cite[Lemma~3.9]{MorelATF} because the description there provides a
  description of $\GW(k)$ as a group. We have included it here to describe the ring
  structure on $\GW(k)$.
\end{rem}

\begin{rem}
  For $L$ over $k$ a finite field extension, the field trace
  $\Tr_{L/k} \colon L \to k$ induces a group homomorphism $\Tr_{L/k}\colon
  \GW(L) \to \GW(k)$ by sending a quadratic form $q \colon V \to L$ to
  $\Tr_{L/k}\circ q\colon V\xrightarrow{q} L \xrightarrow{\Tr_{L/k}} k$. See
  Scharlau~\cite[Chapter~2, Discussion around Lemma~5.5]{ScharlauQHF} for details.
\end{rem}

For a perfect field $k$, Morel \cite[Thm.~6.4.1, Rmk.~6.4.2]{MorelIAHT} constructs an isomorphism $\GW(k)
\to \End_{\SH(k)}(1_k)$ from the Grothendieck-Witt ring of quadratic forms
over $k$ to the endomorphisms
of the unit of $\SH(k)$.

\begin{definition}
  The \emph{quadratic Euler characteristic} of a smooth, projective scheme $X$ is
  the element $\chi(X/k) \in \GW(k)$ corresponding with
  $\chi^{\text{cat}}(\Sigma^\infty X) \in \End_{\SH(k)}(1_k)$ under Morel's
  isomorphism.
\end{definition}

\begin{rem}
  \label{rem:euler-char-properties}
  The quadratic Euler characteristic satisfies a number of expected properties:
  \begin{itemize}
  \item Base change: if $L$ over $k$ is a field extension and $X$ is a smooth,
    projective scheme over $k$, then $\chi(X_L/L) = \chi(X/k)_L\in \GW(L)$; see
    \cite[Proposition~2.4~6.]{LevineA}.
  \item Fibre bundles: if $E \to X$ is a Zariski-locally trivial fibre bundle between smooth,
    projective schemes with smooth, projective fibre $F$, we have $\chi(E/k) =
    \chi(X/k)\cdot \chi(F/k)$; see \cite[Proposition~2.4~1.]{LevineA}
  \item Blow-ups: if $Z \subset X$ is a smooth, closed subscheme of a smooth, projective
    scheme $X$, then $\chi(\Bl_ZX/k) = \chi(X/k) - \chi(Z/k) + \chi(E/k)$ is the
    exceptional divisor. If $Z\subset X$ has constant codimension $c+1$, we can
    use the fibre bundle formula to get $\chi(\Bl_ZX/k) = \chi(X/k)
    (\chi(\PP^c/k) - \langle 1\rangle) \cdot \chi(Z/k)$. See
    \cite[Proposition~2.4~1. and~3.]{LevineA}.
  \end{itemize}
\end{rem}

See \cite[Section~2]{LevineA} for a detailed exposition of the quadratic
Euler characteristic and its basic
properties.

\begin{rem}
  \label{rem:euler-char-rank-signature}
  Let $k \subset \R$ and let $X$ be a smooth, projective scheme over $k$.
  Levine \cite[Remark 2.3]{LevineA} shows that the rank of $\chi(X/k)$ agrees with
  the topological Euler characteristic of $X(\C)$ and the signature
  agrees with the topological Euler characteristic of $X(\R)$.

  If $k$ is not of characteristic $2$ and $X$ is a smooth, projective scheme
  over $k$, a theorem of
  Saito \cite[Theorem 2]{Saito} can be used to interpret the discriminant of
  $\chi(X/k)$ in terms of the determinant of $\ell$-adic cohomology for any
  $\ell$ that is coprime to the characteristic of $k$. A more detailed
  description can be found in \cite[Theorem
  2.22]{Pajwani-PalYZ}.
\end{rem}

\begin{example}
  \label{ex:euler-char-projective-space}
  We have
  \[
    \chi(\PP^n/k) = \sum_{i=0}^n \langle (-1)^i\rangle =
    \begin{cases}
      \langle 1\rangle + \frac n2\cdot H & \text{if $n$ is even},\\
      \frac{n+1}{2} H & \text{if $n$ is odd}.
    \end{cases}
  \]
  See for example \cite[Example 2.6 1.]{LevineA} for a proof.
\end{example}

\begin{example}
  \label{ex:euler-char-curve}
  Let $C$ be a smooth, projective curve of genus $g$ over $k$. Then we have $\chi(C/k) = (1-g)H$.
\end{example}

\begin{rem}
  Further computations of quadratic Euler characteristics include
  \begin{enumerate}
  \item a purely algebraic computation of the quadratic Euler characteristic
    of a smooth,
    same-degree complete intersections in $\PP^n$ due to Viergever
    \cite[Theorem~F and Theorem~G]{Viergever} extending work of Levine-Pepin Lehalleur-Srinivas
    \cite[Theorem~3.12]{LevineLehalleurSrinivas} and
  \item the quadratic Euler
    characteristic of Grassmannians due to Hoyois \cite[Prop.~2.5]{LevineA}, see
    also Brazelton-McKean-Pauli
    \cite[Theorem~8.4]{BrazeltonMcKeanPauli23Bezoutians}.\qedhere
  \end{enumerate}
\end{rem}

\subsection{Euler Classes}
\label{sec:euler-class}
Euler classes, as introduced here, are a motivic analogue of top Chern
classes. They can be used to compute the quadratic Euler characteristic using
the Motivic Gauß-Bonnet Theorem. We provide a more detailed introduction to
Euler classes here because we need the constructions in Section \ref{sec:coverings}.

\begin{definition}[{\cite[Definition 6.3.1]{MorelIAHT}}]
  We denote by $K_\ast^\MW(k)$ the graded associative
  algebra with a generator $[u]$ of degree $+1$ for each $u \in k^\times$ and
  with one generator $\eta$ of degree $-1$ subject to the following relations:
  \begin{itemize}
  \item For $a, b \in k^\times$, we have $[ab] = [a] + [b] + \eta [a][b]$.
  \item (Steinberg relation) For each $a \in k^\times \setminus \{1\}$, we
    have $[a][1-a] = 0$.
  \item For each $u \in k^\times$, we have $\eta[u] = [u]\eta$.
  \item We have $\eta^2[-1] + 2\eta = 0$.
  \end{itemize}
  We call $K^\MW_\ast(k)$ \emph{Milnor-Witt K-theory}.
\end{definition}

\begin{rem}
  There is a canonical isomorphism $\GW(k) \to K^\MW_0(k)$ given by sending
  $\langle a\rangle$ to $1+\eta[a]$, see for example \cite[Lemma
  6.3.8]{MorelIAHT}.

  Note that $K_\ast^\MW(k)/\eta$ is canonically isomorphic
  to Milnor K-theory; see \cite[Remark 6.3.2]{MorelIAHT}. Furthermore,
  $K^\MW_n(k)$ is isomorphic to the Witt ring $\GW(k)/H$ for $n <
  0$ by for example \cite[Remark 6.4.3]{MorelIAHT}.
\end{rem}

\begin{rem}
  One can extend the definition of the Grothendieck-Witt ring and Milnor-Witt K-theory to Nisnevich sheaves
  $\sGW$ and $\K^\MW_\ast$ on the category of smooth schemes over a
  base scheme $X$. The graded sheaf
  $\K^\MW_\ast$ forms a so-called homotopy module and therefore, its sheaf
  cohomology is represented by a spectrum. Using this, one can construct a sheaf
  cohomology with support. See \cite[2192]{LevineA} for details.

  The canonical isomorphism $\GW(k) \to K^\MW_0(k)$ extends to a
  morphism of sheaves $\sGW \to \K^\MW_0$, which is an isomorphism whenever
  $k$ is infinite. This follows from a theorem of Gille-Scully-Zhong
  \cite[Theorem 6.3]{Gille16MW-groups-local}.
\end{rem}

\begin{definition}[see
  {\cite[Section~1.2]{calmès2017finitechowwittcorrespondences}}]
  Let $X$ be a smooth scheme
  and let $\mathcal L$ be a line bundle over $X$. Let $\Z[\mathcal L^\times]$
  be the Nisnevich sheafification of $U \mapsto \Z[L(U)\setminus \{0\}]$. This
  is a module over $\Z[\OO_X^\times]$.

  The assignments $u \mapsto \langle u\rangle$ and $u \mapsto 1 + \eta[u]$
  define $\Z[\OO_X^\times]$-module structures on $\sGW$ and $\K^\MW_0$
  respectively. Furthermore $\K^\MW_n$ is a module over $\K^\MW_0$. We
  define the \emph{Grothendieck-Witt sheaf twisted by $\mathcal L$} as
  \[
    \sGW(\mathcal L) := \sGW\otimes_{\Z[\OO_X^\times]} \Z[\mathcal L^\times],
  \]
  and similarly the \emph{$n$-th Milnor-Witt K-theory sheaf twisted by
    $\mathcal L$} as
  \[
    \K^\MW_n(\mathcal L) := \K_n^\MW \otimes_{\Z[\OO_X^\times]} \Z[\mathcal
    L^\times] = \K_n^\MW\otimes_{\K_0^\MW}\K_0^\MW(\mathcal L).
  \]
  In particular, if $k$ is infinite, we have $\K_0^\MW(\mathcal L) \cong
  \sGW(\mathcal L)$ and $\K^\MW_n(\mathcal L) =
  \K_n^\MW\otimes_{\sGW}\sGW(\mathcal L)$.
\end{definition}

\begin{rem}[{\cite[2194]{LevineA}}]
  \label{rem:twisted-gw-square-iso}
  Let $\mathcal L$ and $\mathcal L'$ be line bundles over $X$. The assignment
  of sending an $\mathcal L$-valued non-degenerate quadratic form $q \colon V \to \mathcal L$
  to the quadratic form $V\otimes \mathcal L' \to \mathcal L\otimes {\mathcal
    L'}^{\otimes 2}$ that in local coordinates is given by $v\otimes \lambda
  \mapsto q(v)\otimes \lambda^2$, induces an isomorphism
  \[
    \psi_{\mathcal L'} \colon \sGW(\mathcal L) \to \sGW(\mathcal L\otimes
    {\mathcal L'}^{\otimes 2}).
  \]
  Similarly, we get also an isomorphism $\psi_{\mathcal L'} \colon
  \K_n^\MW(\mathcal L) \to \K_n^\MW(\mathcal L\otimes
  {\mathcal L'}^{\otimes 2})$.

  Also, if $\rho\colon \mathcal L \to \mathcal L'$ is an isomorphism, we
  get induced isomorphisms $\rho_\ast\colon \sGW(\mathcal L) \to \sGW(\mathcal
  L')$ and $\rho_\ast\colon \K_n^\MW(\mathcal L) \to \K^\MW_n(\mathcal
  L')$.
\end{rem}

By \cite[Proposition 3.3 1.]{LevineA}, Milnor-Witt K-theory is a so-called SL-oriented cohomology
theory. See Ananyevskiy \cite{Ananyevskiy2020SLorient} for a general introduction to
SL-oriented cohomology theories. If one combines this with the six functor formalism, one
can define pushforward maps
\[
  H^{a}(X,\K^\MW_b(\omega_{X/k}\otimes f^\ast(L))) \to
  H^{a-d}(Y,\K^\MW_{b-d}(\omega_{Y/k}\otimes L))
\]
for each proper map $f\colon X \to Y$ of relative
dimension $d$ between smooth $k$-schemes. See \cite[2191]{LevineA} for details.

\begin{definition}
  Let $V\to X$ be a vector bundle of rank $r$ over a smooth scheme $X$. The \emph{Euler
    class} of $V$ is defined as
  \[
    e(V) := s_0^\ast (s_0)_\ast(1_X) \in H^r(X,\K^\MW_r(\det^{-1}V))
  \]
  where $s_0\colon X \to V$ denotes the zero section.
\end{definition}

\begin{rem}
  We can vary the section in the definition of $e(V)$: we also have
  $e(V) = s^\ast(s_0)_\ast(1_X) \in H^r(X,\K_r^\MW(\det^{-1}V))$ for every
  section $s \colon X\to V$. Moreover, whenever
  $Z$ contains $s^{-1}(0)$, this Euler class descends to an \emph{Euler class with
  support} in $Z$
  \[
    e_Z(V,s) \in H^r_Z(X,\K_r^\MW(\det^{-1}V))
  \]
  yielding $e(V)$ under the map forgetting the support. In general, the Euler class
  with support depends on the chosen section. See
  \cite[2191]{LevineA} for details.
\end{rem}

\begin{rem}[{\cite[Theorem 8.1]{LevineA}}]
  \label{rem:euler-class-dual}
  Euler classes satisfy similar properties as top Chern classes. For example,
  let $V$ be a rank $r$ vector bundle over $X$ with dual bundle $V^\vee$. The identification
  $\det^{-1} V^\vee=\det V=\det^{-1}V\otimes (\det V)^{\otimes 2}$ induces an isomorphism
  \[
    \psi\colon H^r(X,\K_r^\MW(\det^{-1}V^\vee)) \to H^r(X,\K_r^\MW(\det^{-1}V)).
  \]
  by Remark \ref{rem:twisted-gw-square-iso}.
  Using this isomorphism, we have
  \[
    \psi(e(V^\vee)) = (-1)^re(V) \in H^r(X,\K_r^\MW(\det^{-1}V)).\qedhere
  \]
\end{rem}

Quadratic Euler characteristics are in general hard to compute.
The motivic Gauß-Bonnet Theorem of Levine-Raksit \cite[Theorem 5.3 and Theorem
8.4]{LevineGB} provides a
way to compute the quadratic Euler characteristic using pushforwards of Euler classes. A more general version has been proved by
Déglise-Jin-Khan \cite[Theorem 4.6.1]{DegliseFCMHT}.

\begin{theorem}[Motivic Gauß-Bonnet, {\cite[Theorem 8.4]{LevineGB} and \cite[Theorem
    4.6.1]{DegliseFCMHT}; see \cite[Theorem 4.1]{LevineA} for this version}]
  Let $\pi\colon X\to \Spec k$ be a smooth, projective $d$-dimensional
  $k$-scheme with tangent bundle $T_X\to X$. Then we have
  \[
    \chi(X/k) = \pi_\ast(e(T_X)) \in GW(k).
  \]
\end{theorem}

\begin{corollary}[{\cite[Corollary 8.7]{LevineGB}}]
  \label{cor:motivic-gb}
  Let $X$ be an odd-dimensional, smooth, projective scheme over $k$.
  Then $\chi(X/k) = m\cdot H$ for some integer $m \in \Z$.
\end{corollary}

In order to use this for concrete computations, one first
needs to compute the Euler class and the pushforward. In the following, we
describe how to compute the Euler class in nice cases, based on \cite[Section~5]{LevineA}. Afterwards, we apply this to get a description of the quadratic
Euler characteristic.

Using the purity isomorphism, one obtains an isomorphism
\[
  H^d_x(X,\K_d^\MW(L)) \cong \GW(k(x), \det^{-1}\m/\m^2 \otimes L)
\]
for $X$ a smooth scheme of dimension $d$ over $k$ and $x \in X$ a
closed point.

If $V\to X$ is a rank $d$ vector bundle over $X$ and $s\colon X \to V$ is a
section with only isolated zeroes $x_1, \dots, x_n$, then we can write by first using
excision and then purity
\[
  H^d_{s^{-1}(0)}(X,\K_d^\MW(L)) \cong \bigoplus_{i=1}^n
  H^d_{x_i}(X,\K_d^\MW(L)) \cong  \bigoplus_{i=1}^n\GW(k(x_i),
  \det^{-1}\m/\m^2 \otimes L).
\]
Locally around every zero $x$ of $s$, the section $s$ gives rise to \emph{local
  Euler classes}
\[
  e_x(V,s) \in H^d_x(X,\K_d^\MW(\det^{-1}V)) \cong \GW(k(x), \det^{-1}\m/\m^2
  \otimes \det^{-1}V).
\]
The Euler class of $V$ is now given by the image of $\sum_{i=1}^n
e_{x_i}(V,s) \in H^d_{s^{-1}(0)}(X,\K_d^\MW(\det^{-1} V))$ under the map
\[
  H^d_{s^{-1}(0)}(X,\K_d^\MW(\det^{-1} V)) \to H^d(X,\K_d^\MW(\det^{-1} V))
\]
forgetting the support.

The pushforward to $\Spec k$ also factors over the local Euler classes; and for these, we have an explicit description of the pushforward as
the trace map: if $V$ is a vector bundle with an isomorphism $\rho\colon
\det^{-1}V \to L^{\otimes 2}\otimes \omega_{X/k}$, then the pushforward to
$\Spec k$ is
given by the trace
\[
  \Tr_{k(x)/k} \colon \GW(k(x)) \to \GW(k)
\]
after identifying $\GW(k(x)) \cong \GW(k(x), \det^{-1}\m/\m^2\otimes \det^{-1}V)$
using $\rho$ as in Remark \ref{rem:twisted-gw-square-iso}.
If $V = \Omega^1_{X/k}$ is the sheaf of Kähler differentials, the resulting
element $p_{X\ast}(e(\Omega^1_{X/k}))\in \GW(k)$ is $(-1)^{\dim X}\chi(X/k)$,
where $p_X\colon X \to \Spec k$ denotes the structure
map. Indeed, by Remark \ref{rem:euler-class-dual}, we have
$\psi(e(\Omega^1_{X/k}))=\psi(e(T_X^\vee))=(-1)^{\dim X}e(T_X)$, where $\psi$
denotes the isomorphism in Remark \ref{rem:euler-class-dual} for the tangent
bundle $T_X$ of $X$. Thus, we have $p_{X\ast}(e(\Omega^1_{X/k}))=(-1)^{\dim X}p_{X\ast}(e(T_X)) =
(-1)^{\dim X}\chi(X/k)$.

Local Euler classes can be computed explicitly using the following
construction due to Scheja and Storch
\cite{SchejaStorch75Spur,SchejaStorch79Residuen}, which has been
refined
by Bachmann and Wickelgren \cite{BachmannWickelgren23SixFF}. The construction
given here can be found in \cite[Section 5]{LevineA}.
  
Let $\OO$ be a regular local ring with residue field $F$ and maximal ideal
$\m$. Assume that $\OO$ contains a field. Let $t_\ast = t_1, \dots, t_n$ be a system of parameters for $\OO$ and
let $s_\ast := s_1, \dots, s_n$ be elements of $\m$ such that the ideal $(s_1,
\dots, s_n)$ is $\m$-primary. Let $J(s_\ast) := \OO/(s_1, \dots, s_n)$. In particular,
the quotient map $J(s_\ast) \to F$ splits since $J(s_\ast)$ is complete. Such
a splitting induces an $F$-algebra structure on $J(s_\ast)$, under which
$J(s_\ast)$ is a finite dimensional $F$-algebra.

Since the ideal $(s_1, \dots, s_n)$ is $\m$-primary, we can find elements
$a_{ij} \in \OO$ satisfying
\[
  s_i = \sum_{j=1}^n a_{ij}t_j
\]
for $i = 1, \dots, n$.

\begin{definition}
  \label{def:scheja-storch}
  We define the \emph{Scheja-Storch element}
  as $e_{t_\ast,s_\ast} := \det(a_{ij}) \in J(s_\ast)$.
\end{definition}

The Scheja-Storch element satisfies a number of properties.

\begin{theorem}[{\cite[Theorem 5.1]{LevineA}}]
  \begin{enumerate}
  \item $e_{t_\ast, s_\ast} \in J(s_\ast)$ is independent of the choice of
    $a_{ij}$.
  \item The socle of $J(s_\ast)$, that is $\{x \in J(s_\ast)\mid \m\cdot x =
    0\}$, is a one-dimensional $k$-vector space with generator
    $e_{t_\ast,s_\ast}$.
  \item Let $\Tr\colon J(s_\ast) \to F$ be an $F$-linear map such that
    $\Tr(e_{t_\ast, s_\ast}) = 1$. Then the bilinear form on $J(s_\ast)$
    \[
      B_{s_\ast,t_\ast}(x,y) := \Tr(xy)
    \]
    is non-degenerate, and $[B_{s_\ast,t_\ast}] \in \GW(F)$ is independent of
    the choice of $\Tr$ (satisfying $\Tr(e_{t_\ast,s_\ast}) = 1$).
  \end{enumerate}
\end{theorem}

\begin{rem}
  There is a more involved construction of the Scheja-Storch element, detailed
  in \cite[Section 5]{LevineA}. This construction also provides an explicit way of
  constructing an $F$-linear map $\Tr\colon J(s_\ast) \to F$ satisfying
  $\Tr(e_{t_\ast,s_\ast}) = 1$.

  The more involved construction of $e_{t_\ast,s_\ast}
  \in J(s_\ast)$ is equivalent to the construction described above by \cite[Theorem
  5.1 5.]{LevineA}. The construction above is more amenable to the
  computations in Section \ref{sec:coverings}.
\end{rem}

\begin{rem}
  In the discussion in \cite[Section 5]{LevineA}, it is assumed that the
  map to the residue field $\OO \to F$ splits. In the construction presented
  here, this is only necessary to obtain an $F$-algebra structure on
  $J(s_\ast)$ so that there is an $F$-linear map $J(s_\ast) \to F$ sending
  $e_{t_\ast,s_\ast}$ to $1$. But for this, it is enough to have a splitting of
  the quotient map $J(s_\ast) \to F$, which always exists.
\end{rem}

We can use this construction to compute local Euler classes.

\begin{theorem}[{\cite[Corollary 5.3]{LevineA}}]
  \label{thm:local-euler-class-ss}
  Let $p\colon V \to X$ be a rank $d$ vector bundle over a
  $d$-dimensional, smooth scheme $X$ over $k$ and let $s\colon X \to V$ be a
  section. Suppose a closed point $x \in X$ is an isolated zero of $s$. Choose a
  framing $e_1, \dots, e_d$ for $V$ in a neighbourhood of $x$ and let $t_\ast
  := t_1, \dots, t_d$ be a system of parameters for the maximal ideal $\m_x
  \subset \OO_{X,x}$. Write $s= \sum_{i=1}^d s_ie_i$ near $x$ and let $s_\ast
  = s_1, \dots, s_d$. Then
  \[
    e_x(V,s) \in H^d_x(X,\K_d^\MW(\det^{-1}V)) = \GW(k(x),
    \det^{-1}\m/\m^2\otimes \det^{-1}V)
  \]
  is given by
  \[
    e_x(V,s) = [B_{s_\ast,t_\ast}] \otimes \frac{\partial}{\partial t_1}\wedge
    \dots \wedge \frac{\partial}{\partial t_d} \otimes (e_1 \wedge \dots
    \wedge e_d)^{-1}.
  \]
\end{theorem}

Note that for $x \in X$ a closed point, we have
$\det^{-1}\m/\m^2=\omega_{X/k}^\vee\otimes_{\mathcal{O}_X}k(x)$. With this
observation, we get an explicit formula for the local Euler classes of a
section of $\Omega^1_{X/k}$.

\begin{corollary}
  \label{cor:local-euler-class-ss-differentials}
  Let $L$ be a line bundle over a $d$-dimensional, smooth scheme $X$ over $k$
  and let $s\colon X \to \Omega^1_{X/k}\otimes L^{\otimes 2}$ be a
  section. Suppose a closed point $x \in X$ is an isolated zero of $s$. Let $l$ be
  a local generator of $L$ around $x$ and let $t_\ast := t_1, \dots, t_d$ be a system
  of parameters for the maximal ideal $\mathfrak m_x \in \OO_{X,x}$. Write $s
  = \sum_{i=1}^d s_i \cdot dt_i\otimes l^{\otimes 2}$ near $x$ and let $s_\ast:=
  s_1, \dots, s_d$. Then we have
  \[
    e_x(\Omega^1_{X/k}\otimes L^{\otimes 2},s) = [B_{s_\ast,t_\ast}] \in \GW(k(x))
  \]
  under the canonical identification
  \[
    \GW(k(x), \det^{-1}\m/\m^2\otimes \det^{-1}\m/\m^2\otimes
    (L^{-1})^{\otimes 2d}) \cong \GW(k(x))
  \]
  from Remark \ref{rem:twisted-gw-square-iso}.

  \begin{proof}
    Note that $\Omega^1_{X/k} = T_{X/k}^\vee$ and under this identification, we
    have $(dt_1\wedge\dots\wedge dt_n)^{-1}=\partial/\partial
    t_1\wedge\dots\wedge \partial/\partial t_n$. Thus, Theorem
    \ref{thm:local-euler-class-ss} yields
    \[
      e_x(V,s) = [B_{s_\ast,t_\ast}] \otimes \frac{\partial}{\partial t_1}\wedge
      \dots \wedge \frac{\partial}{\partial t_d} \otimes \frac{\partial}{\partial t_1}\wedge
      \dots \wedge \frac{\partial}{\partial t_d}\otimes l^{\otimes 2d}
    \]
    and the twist
    \[
      \frac{\partial}{\partial t_1}\wedge
      \dots \wedge \frac{\partial}{\partial t_d} \otimes \frac{\partial}{\partial t_1}\wedge
      \dots \wedge \frac{\partial}{\partial t_d}\otimes l^{\otimes 2d}
    \]
    gets eliminated under the identification $\GW(k(x),
    \det^{-1}\m/\m^2\otimes \det^{-1}\m/\m^2\otimes \det^{-1}
    L^{\otimes 2}) \cong \GW(k(x))$.
  \end{proof}
\end{corollary}

Together with the motivic Gauß-Bonnet theorem, one can use this to compute
quadratic Euler characteristics in the following situation: let $f\colon Y\to X$ be a projective morphism from a smooth, projective,
integral, $r$-dimensional $k$-scheme $Y$ to a smooth, projective curve $X$ over $k$.
If $r$ is odd, we also assume that $X$ has a half-canonical
line bundle $M$, that is a line bundle $M$ satisfying $\omega_{X/k} \cong
M^{\otimes 2}$.
One can relate
the Euler class of $\Omega^1_{Y/k}$ to the Euler class of $\Omega^1_{X/k}$,
see \cite[Proposition 11.1]{LevineA}. Using this, one can compute the
quadratic Euler characteristic of $Y$ using the quadratic Euler characteristic
of $X$ and some properties of $f$, see \cite[Theorem 11.2]{LevineA}.

A
\emph{critical point} of $f$ is a point $y \in Y$ with $df(y) = 0$, a
\emph{critical value} of $f$ is a point $x = f(y)$ of $X$ with $y$ a critical
point. We denote the set of all critical points of $f$ by
$\critical(f)$.

For the remainder of this section, we assume that $f$ only has finitely
many critical points.

\begin{definition}
  Let $f\colon Y\to X$ be a morphism from a smooth, projective, integral
  $k$-scheme $Y$ of dimension $r$ to a smooth, projective curve $X$.
  Let $y$ be a critical point of $f$. Then we define
  \[
    e_y(df) := e_y(\Omega_{Y/k}\otimes f^\ast\omega^{-1}_{X/k};df) \in
    H^r_y(Y,K_r^\MW(\omega_{Y/k}^{-1}\otimes f^\ast\omega^{\otimes r}_{X/k})).\qedhere
  \]
\end{definition}

\begin{rem}[{\cite[pp. 2225f]{LevineA}}]
  There is a comparison isomorphism
  \begin{equation}
    \label{eq:critical-euler-class-comparison-iso}
    K_0^\MW(y) \cong H^r_y(Y,\K_r^\MW(\omega_{Y/k})) \cong
    H^r_y(Y,K_r^\MW(\omega_{Y/k}^{-1}\otimes f^\ast\omega^{\otimes r}_{X/k})).
  \end{equation}
  If $r$ is even,
  the first isomorphism in \eqref{eq:critical-euler-class-comparison-iso} is the purity isomorphism and the
  second isomorphism is induced by $\psi_{\omega_{Y/k}^{-1}\otimes f^\ast\omega^{\otimes r/2}_{X/k}}$
  in Remark \ref{rem:twisted-gw-square-iso}.

  If $r$ is odd, we can construct the comparison isomorphism \eqref{eq:critical-euler-class-comparison-iso} using a
  fixed isomorphism $\rho \colon \omega_{X/k} \to M^{\otimes 2}$ to a
  half-canonical line bundle $M$. In this situation, the first isomorphism in
  \eqref{eq:critical-euler-class-comparison-iso} is again the purity
  isomorphism and the second isomorphism is induced by
  $(\id_{\omega_{Y/k}^{-1}}\otimes f^\ast(\rho^{\otimes r})^{-1})_\ast \circ
  \psi_{\omega_{Y/k}^{-1}\otimes f^\ast M^{\otimes r}}$ in Remark
  \ref{rem:twisted-gw-square-iso}.

  We will consider $e_y(df)$ to be an element in $K_0^\MW(y) =
  \GW(k(y))$ by applying \eqref{eq:critical-euler-class-comparison-iso}.
\end{rem}

\begin{rem}
  In the situation considered here, we also have the pushforward
  \[
    i_{y\ast}\colon K_0^\MW(y) \to H^r(Y,\K_r^\MW(\omega_{Y/k}))
  \]
  forgetting the support.
\end{rem}

\begin{theorem}[{\cite[Corollary 11.4]{LevineA}}]
  \label{thm:riemann-hurwitz}
  Let $f\colon Y \to X$ be a projective morphism between a smooth, projective,
  integral scheme $Y$ of dimension $r$ over $k$ and a smooth, projective
  curve $X$ over $k$. Suppose $f$ has only finitely many critical points. In
  addition, suppose that $X$ admits a half-canonical line bundle in case $r$ is
  odd. Then
  \[
    (-1)^r\chi(Y/k) = \sum_{y \in \critical(f)} \pi_{Y\ast}i_{y\ast}e_y(df) -
    D(f)\cdot H
  \]
  in $\GW(k)$. Here
  \[
    D(f) = \frac{1}{2}[\deg_k(c_r(\Omega_{Y/k}\otimes
    f^\ast\omega^{-1}_{X/k})) - \deg_k(c_r(\Omega_{Y/k}))] \in \Z
  \]
  and $c_r$ is the $r$-th Chern class in the Chow ring of $Y$.
\end{theorem}

\begin{rem}
  In Theorem \ref{thm:riemann-hurwitz}, we have
  \[
    D(f) = (-1)^{r-1}\cdot \frac 12 \cdot \chi^\topo(Y_t)\cdot \chi^\topo(X)
  \]
  where $Y_t$ is any smooth fibre of $f$. This follows essentially by
  \cite[Proof of Theorem~11.2]{LevineA} and a Gauß-Bonnet formula for the
  topological Euler characteristic. This yields another argument why $D(f)$ is
  an integer: since $X$ is odd-dimensional, we have that $\chi^\topo(X)$ is
  even.

  The same also holds in positive characteristic after replacing the topological
  Euler characteristic with the étale Euler characteristic.
\end{rem}

Consider a critical point $y \in Y$ and $x = f(y) \in X$ and denote the
maximal ideal of the local ring $\OO_{X,x}$ at $x$ by $\m_x$. A parameter
$t_x \in \m_x$ is \emph{normalised} if there exists a generating section
$\lambda_{M,x}$ of $M$ around $x \in X$ such that
\[
  \rho^{\otimes -1}(\partial/\partial t_x) = \lambda_{M,x}^{-2}\otimes k(x)
\]
via the canonical identification $\omega^{-1}_{X/k} \otimes k(x) \cong (\m_x/\m_x^2)^\vee$.
See \cite[Discussion before Corollary~11.6, p.~2227]{LevineA}.

By applying Theorem \ref{thm:local-euler-class-ss}, one obtains the following
corollary.

\begin{corollary}[{\cite[Corollary 11.6]{LevineA}}]
  \label{cor:riemann-hurwitz-ss}
  Let $f\colon Y \to X$ be a projective morphism from a smooth, projective,
  integral scheme $Y$ of dimension $r$ over $k$ and a smooth, projective
  curve over $k$. Suppose $f$ has only finitely many critical points. In
  addition suppose that $X$ admits a half-canonical line bundle in case $r$ is
  odd.

  For each $y \in \critical(f)$, we choose a system of parameters $t_1^y,
  \dots, t_r^y \in \m_y$ and a parameter $t_x \in \m_x$, where $x = f(y)$; if
  $r$ is odd, we assume that $t_x$ is normalised. Write
  \[
    d(f^\ast(t_x)) = \sum_{i=1}^rs_i^y\cdot dt_i^y
  \]
  with $s_i^y \in \OO_{Y,y}$. This yields a class of the Scheja-Storch form
  $[B_{s_\ast^y,t_\ast^y}]\in \GW(k(y))$. Then
  \[
    (-1)^r\chi(Y/k) = \sum_{y \in \critical(f)} \Tr_{k(y)/k}([B_{s_\ast^y,t_\ast^y}]) -
    D(f)\cdot H
  \]
  in $\GW(k)$.

  Moreover, we have $e_y(df) = [B_{s_\ast^y,t_\ast^y}] \in \GW(k(y))$ and $\pi_{Y\ast}i_{y\ast}e_y(df) =
  \Tr_{k(y)/k}([B_{s_\ast^y,t_\ast^y}]) \in \GW(k)$ for all $y \in \critical(f)$.
\end{corollary}

\begin{definition}
  \label{def:a1-milnor-no}
  We define the \emph{$\A^1$-Milnor number} of $f$ at a critical point $y \in
  Y$ to be $\mu^{\A^1}_y(f) := e_y(df)$ considered as an element of $\GW(k(y))$.
\end{definition}

\begin{rem}
  Definition \ref{def:a1-milnor-no} is an extension of the $\A^1$-Milnor
  number as defined by Kass-Wickelgren \cite[Definition
  35]{KassWickelgren2019EKLdegree}. Indeed, by
  \cite[Main Theorem]{KassWickelgren2019EKLdegree}, the $\A^1$-Milnor number as defined
  in \cite[Definition 35]{KassWickelgren2019EKLdegree} is computed by the same
  quadratic form $B_{s_\ast^y,t_\ast^y}$ as in Corollary
  \ref{cor:riemann-hurwitz-ss}.
\end{rem}

\begin{rem}
  \label{rem:milnor-no-rank}
  If $k = \C$, the rank of $\mu^{\A^1}_y(f)$ agrees with the classical
  Milnor number of $f$ at $y$. See \cite[Remark
  37]{KassWickelgren2019EKLdegree} for details. Thus, it makes sense to define
  $\mu_y(f) := \rank\mu^{\A^1}_y(f)$ over an arbitrary field $k$.
\end{rem}

\section{Geometrically Cyclic Branched Coverings}
\label{sec:coverings}
Geometrically cyclic branched coverings are an algebro-geometric analogue of
branched coverings in topology. In topology, there is a Riemann-Hurwitz formula relating the
topological Euler characteristics of the spaces involved in a branched
covering. In this section, we shall introduce geometrically cyclic branched
coverings and we shall analyse whether
there is a similar Riemann-Hurwitz formula in such a situation. This analysis
splits up into two parts: Section \ref{sec:etale-classes} deals with the
contribution from the étale part of such a covering and Section
\ref{sec:odd-contribution} deals with the
contribution from the branch locus.

Throughout this section, let $n\ge 1$ be an integer that is invertible in
$k$.

\label{sec:coverings-def}
\begin{definition}
  An \emph{$n$-fold geometrically cyclic covering branched at $Z$}
  is a map between smooth,
  projective schemes $f\colon Y \to
  X$ over $k$ together with a smooth, closed subscheme $Z\subset X$ satisfying
  the following condition: there exists a line bundle $L \to X$,
  a section $s \colon X \to L^{\otimes n}$ with zero subscheme $Z\subset
  X$, and a morphism $Y \to L$ such that we have a fibre square
  \[
    \begin{tikzcd}
      Y \ar[r]\ar[d, "f"] & L\ar[d, "m_n"]\\
      X \ar[r, "s"] & L^{\otimes n}.
    \end{tikzcd}
  \]
  Here $m_n\colon L \to L^{\otimes n}$ is the map sending a local generator
  $t$ of $L$ to $t^n$.

  We call $Z\subset X$ the \emph{branch locus} of $f$.
\end{definition}

\begin{rem}
  The name \emph{geometrically} cyclic covering is chosen since an $n$-fold
  geometrically cyclic covering becomes a
  Galois covering with Galois group $\Z/n\Z$ after base-changing to an
  algebraically closed field. That is, we have for $Y \to X$ an $n$-fold geometrically
  cyclic covering over an algebraically closed field $k$ that $\Aut_X(Y) \cong
  \Z/n\Z$ and $Y/\Aut_X(Y) = X$.
\end{rem}

\begin{rem}
  \label{rem:ramification-factor}
  The closed immersion $Z\to X$ of the branch locus factors through the
  covering morphism $Y\to X$. This
  yields a canonical closed immersion $Z \to Y$.
  Thus, we can consider points in $Z$ as points in $X$ and as points
  in $Y$. In the following, we will consider $x \in Z$ as $x \in X$ or $x \in
  Y$ without writing the inclusion morphisms.
\end{rem}

\begin{rem}
  \label{rem:etale-part}
  Note that the induced morphism on open subschemes $f|_{Y\setminus Z}\colon Y\setminus Z \to
  X\setminus Z$ is étale.
\end{rem}

\begin{rem}
  This is an algebraic version of an $n$-fold covering of topological spaces
  $\pi\colon Y \to X$ branched at $A$, i.e. $\pi\colon
  Y\setminus \pi^{-1}(A) \to X\setminus A$ is an $n$-fold covering of
  topological spaces and $\pi\colon
  \pi^{-1}(A) \to A$ is a homeomorphism, where $(X,A)$ is a suitably nice
  pair of topological spaces, such as a CW pair. The topological Euler characteristic now satisfies
  \[
    \chi^\topo(Y) = n\cdot \chi^\topo(X) - (n-1)\cdot \chi^\topo(A).\qedhere
  \]
  For geometrically cyclic branched coverings, the same relation also holds in
  for the étale Euler characteristic.
\end{rem}

\begin{rem}
  One applies this classically to a branched covering of curves $Y\to X$ to
  yield the Hurwitz formula relating the topological Euler characteristics of
  $Y$ and $X$.
\end{rem}

For computing the quadratic Euler characteristic, we consider the following situation: for the remainder of this
section, fix an $n$-fold geometrically cyclic covering $f\colon Y \to X$ branched
at $Z$. Let $\rho\colon X \to C$ be a
morphism to a smooth, projective curve $C$ whose canonical bundle is a square,
for example $C = \PP^1$. Denote
the closed immersion $Z \to X$ by $i$. Assume that $\rho$ has only finitely many critical points, all of
which are outside of $Z$, that the composition
$\rho\circ i$ also has only finitely many critical points, and the $n$ is
invertible in $k$.

\begin{rem}
  The assumption of having a morphism $\rho\colon X \to C$ as above is
  mild. Indeed as detailed in Appendix \ref{sec:lefschetz}, one can find a
  sufficiently nice, codimension two subscheme $L$ of $X$ such that by blowing up $X$ along $L$, we obtain a morphism $\rho\colon
  \Bl_LX \to \PP^1$ satisfying the above properties and such that the induced
  morphism $\Bl_{f^{-1}(L)}Y\to \Bl_LX$ is still an $n$-fold geometrically
  cyclic covering branched at the smooth, closed subscheme
  $\Bl_{Z\cap L}Z \subset \Bl_LX$. See Remark \ref{rem:map-construction} below for
  details.
\end{rem}

Recall from Section \ref{sec:euler-class} that we denote the set of critical
points of $f \colon Y \to X$ by $\critical(f) \subset Y$.

\begin{proposition}
  \label{prop:critical-decomp}
  Let $n \ge 2$. Then we have
  \[
    \critical(\rho\circ f) = \critical(\rho\circ i) \cup f^{-1}(\critical(\rho)).
  \]
  Moreover, this union is disjoint.

  \begin{proof}
    The union is disjoint since $\critical(\rho)\cap Z = \emptyset$. Thus, it
    remains to show that the equality holds.

    We start with the inclusion ``$\supset$''. We have $f^{-1}(\critical(\rho)) \subset
    \critical(\rho\circ f)$ by functoriality of taking differentials.

    Let $x \in \critical(\rho\circ i)$. We need to
    show $x \in \critical(\rho\circ f)$. Pick a trivialisation of $L$ around
    $x$. Let $A$ be the local ring of $X$ at
    $x$, let $t_0 \in A$ be the element corresponding with the section
    $s\colon X\to L^{\otimes n}$ under the chosen trivialisation of $L$, and let $t_0, t_1, \dots, t_r$ be a set of local parameters of
    $A$. Therefore, $A/(t_0)$ is the local ring of $Z$ at $x$ with set
    of parameters $\bar t_1, \dots, \bar t_r$ and $A[z]/(z^n-t_0)$ is the local ring of
    $Y$ at $x$ with set of parameters $z, t_1, \dots, t_r$. Let $dt\in
    \omega_{C/k,\rho(x)}$ be a generator. Hence, we can write
    \[
      d\rho(dt) = \sum_{i=0}^r s_idt_i.
    \]
    Since $d(\rho\circ i)$ varnishes, we have $s_i \in (t_0, \dots,
    t_r)\subset A$ for $i =1, \dots, r$. Over $A[z]/(z^n-t_0)$, we can write
    \[
      d(\rho\circ f)(dt) = ns_0z^{n-1}dz + \sum_{i=1}^rs_idt_i.
    \]
    Since $z^n = t_0$ in $A[z]/(z^n-t_0)$, we have $(t_0, \dots, t_r) \subset
    (z, t_1, \dots, t_r) \subset A[z]/(z^n-t_0)$. In particular because $n \ge
    2$, we have
    $ns_0z^{n-1}, s_1, \dots, s_r \in (z, t_1, \dots, t_r)$, which is the
    maximal ideal of $A[z]/(z^n-t_0)$. And thus, $d(\rho\circ f)(dt)$ vanishes,
    which implies $d(\rho\circ f) = 0$. Thus, $x \in \critical(\rho\circ f)$
    as desired.

    For the other inclusion, let $x \in \critical(\rho\circ f)$. Note that
    $Y\setminus Z \to X\setminus Z$ is étale by Remark \ref{rem:etale-part}. Thus, if $x \in Y\setminus Z$, we have that
    $df_x$ is invertible and thus, we have
    \[
      0 = (df_x)^{-1}\circ d(\rho\circ f)_x=
      (df_x)^{-1} \circ df_x \circ d\rho_{f(x)} = d\rho_{f(x))}.
    \]
    Thus, $f(x)
    \in \critical(\rho)$, which implies $x \in f^{-1}(\critical(\rho))$.

    If $x \in Z$, note that $i = f \circ i'$ where $i'\colon Z \to Y$ is
    the closed immersion from Remark~\ref{rem:ramification-factor}. Thus, we have
    \[
      d(\rho\circ i)_x = d(\rho\circ f\circ i')_x= di'_x\circ  d(\rho\circ f)_x
      = 0.
    \]
    And thus, $x \in \critical(\rho\circ i)$.
  \end{proof}
\end{proposition}

\begin{rem}
  Note that Proposition \ref{prop:critical-decomp} fails for $n =1$: in this
  case, the covering map $f\colon Y \to X$ is the identity and we have
  $\critical(\rho\circ f) = f^{-1}(\critical(\rho))$.
\end{rem}

Morally speaking, Proposition \ref{prop:critical-decomp} tells us that the critical points of
$\rho\circ f$ are of two kinds: there are the critical points coming from $Z$
and there are the critical points coming from $X$. We call the critical points
contained in $\critical(\rho\circ i)$ the \emph{critical points from the
  branch locus}
and we call the critical points contained in $f^{-1}(\critical(\rho))$ the \emph{étale critical
  points}. In the following, we will deal with both kinds of critical points
separately.

\begin{rem}\label{rem:induced-kaehler-section}
  A morphism $\rho \colon X \to C$ from a smooth, projective $k$-scheme $X$ to a
  smooth, projective curve $C$ induces a map
  $d\rho\colon \rho^\ast\omega_{C/k}\to \Omega_{X/k}$, which is the same as a
  section $\omega_\rho$ of $\Omega_{X/k}\otimes\rho^\ast\omega_{X/k}^{-1}$. In
  particular, $d\rho$ is zero at some point $x \in X$ if and only if the
  $\omega_\rho$ is zero at $x$.
\end{rem}

\subsection{The Étale Contribution}
\label{sec:etale-classes}
The following lemma is a special case of \cite[Chapter 2, Theorem 5.6]{ScharlauQHF}. We include
a proof here for the convenience of the reader.

\begin{lemma}
  \label{lem:trace-quadratic-linear}
  Let $L/k$ be a finite field extension. Then
  \[
    \Tr_{L/k}\colon \GW(L) \to \GW(k)
  \]
  is $\GW(k)$-linear with respect to the $\GW(k)$-module structure on $\GW(L)$
  induced by the pushforward $\GW(k) \to \GW(L)$.

  \begin{proof}
    The trace map is additive by \cite[Chapter 2, Discussion after Lemma~5.5]{ScharlauQHF}. Thus, we only need to
    prove the compatibility with scalar multiplication. For this,
    it is enough to consider generators; that is, we have to show that
    $\Tr_{L/k}(\langle a\rangle \cdot \langle b\rangle) = \langle a\rangle
    \cdot \Tr_{L/k}(\langle b\rangle)$ for $a \in k^\times$ and $b \in L^\times$.

    Let $x,y \in L$. Then we have
    \[
      \Tr_{L/k}(\langle a\rangle \cdot \langle b\rangle)(x,y) =
      \Tr_{L/k}(abxy) = a\Tr_{L/k}(bxy) = (\langle a\rangle\cdot
      \Tr_{L/k}(\langle b\rangle))(x,y).\qedhere
    \]
  \end{proof}
\end{lemma}

\begin{notation}
  If $A$ is a finite $k$ algebra, we will denote the quadratic form associated with the
  symmetric bilinear form
  \[
    A\times A \to k; (x,y) \mapsto \Tr_{A/k}(x\cdot y)
  \]
  by $\Tr_{A/k}(\langle 1\rangle)$. This is a canonical extension of the
  trace of quadratic forms over fields.
\end{notation}

\begin{lemma}
  \label{lem:cyclic-trace}
  Let $s \in k$ and assume that $n$ is invertible
  in $k$. Let $A = k[z]/(z^n-s)$. Then $\Tr_{A/k}(\langle 1\rangle)$ is
  non-degenerate and
  \[
    \Tr_{A/k}(\langle 1\rangle) =
    \begin{cases}
      \langle n \rangle + \frac{n-1}{2}\cdot H & \text{if $n$ is odd},\\
      \langle n\rangle + \langle ns\rangle + \frac{n-2}{2}\cdot H& \text{if
                                                                   $n$ is even.}
    \end{cases}
  \]

  \begin{proof}
    The $k$-algebra $A$ has $k$-basis $1, z^1,
    \dots, z^{n-1}$. Note that for $0 \le i \le 2n-2$, we have
    \[
      \Tr_{L/k}(z^i) =
      \begin{cases}
        n & \text{if }i = 0,\\
        ns& \text{if }i = n,\\
        0 & \text{otherwise}.
      \end{cases}
    \]
    Therefore, we get the following diagonal matrix representing
    $\Tr_{A/k}(\langle 1\rangle)$ if $n$ is odd: we have a diagonal entry
    $n$ for the basis element $1$ and for every pair $(z^i, z^{n-i}), i = 1,
    \dots, \frac{n-1}{2}$ a block matrix
    \[
      \begin{pmatrix}
        0 & ns\\
        ns & 0
      \end{pmatrix}
    \]
    on the diagonal. In particular, $\Tr_{A/k}(\langle 1\rangle)$ is
    non-degenerate and thus represents an element of $\GW(k)$. Each two by two block represents $H$ in $\GW(k)$, and thus, we get
    the statement for $n$ odd.

    For $n$ even, we can run the same computation when noting that we get an
    additional $ns$ on the diagonal for the basis element $z^{n/2}$ and we
    therefore get a hyperbolic form contribution for every pair
    $(z^i,z^{n-i})$ with $i = 1, \dots, \frac{n-2}{2}$, plus one copy of the
    rank one form $\langle ns\rangle$.
  \end{proof}
\end{lemma}

\begin{lemma}
  \label{lem:etale-scheja-storch-change}
  Let $\varphi\colon A \to B$ be a homomorphism of local rings. Assume that
  the induced map
  $\Spec\varphi\colon \Spec B \to \Spec A$ is
  étale. Let $s_1, \dots, s_r\in A$ be $\mathfrak m$-primary with $\mathfrak m
  \subset A$ the maximal ideal. Let $t_1, \dots, t_r$ be local parameters for
  $A$. Then $\varphi(t_1), \dots, \varphi(t_r)$ are local parameters for $B$
  and $\varphi(s_1), \dots, \varphi(s_r)$ are $\m'$-primary with $\m'\subset
  B$ the maximal ideal.

  Let $e_{s_\ast, t_\ast} \in J(s_\ast) = A/(s_1, \dots, s_r)$ and
  $e_{\varphi(s_\ast),\varphi(t_\ast)} \in J(\varphi(s_\ast)) =
  B/(\varphi(s_1), \dots, \varphi(s_r))$ be the Scheja-Storch generators. Then
  we have $\varphi(e_{s_\ast, t_\ast}) =
  e_{\varphi(s_\ast),\varphi(t_\ast)}$.

  \begin{proof}
    Let $\mathfrak m'$ be the maximal ideal of $B$.
    Since $\Spec\varphi$ is étale, we have that $\mathfrak m'/\mathfrak m'^2 =
    \mathfrak m/\mathfrak m^2\otimes_{A/\mathfrak m}B/\mathfrak m'$. In
    particular, Nakayama's Lemma yields that $\varphi(t_1),\dots,
    \varphi(t_r)$ are local parameters for $B$.

    The ideal $I = (\varphi(s_1), \dots, \varphi(s_r))$ is $\m'$-primary
    since the radical ideal of $I$ is $\m'$,
    which is a maximal ideal, see for example \cite{StackExchangePrimary}. We include this short argument here for
    the convenience of the reader. Let $x,y \in B$ with $xy \in I$ and assume
    that $y \not\in \sqrt{I} = \m'$. Then we have to show that $x \in
    I$. Since $\m'$ is maximal, we can find $m \in \m'$ and $b \in B$ with $by
    + m = 1$. Since $\sqrt I = \m'$, there is some $n \ge 1$ with $m^n \in I$
    and thus, we have $1 = 1^n = (by+m)^n = b'y + m^n$ for some $b' \in
    B$. Therefore, we have $x = x(b'y + m^n) = b'xy + xm^n \in I$. Hence, $I$ is
    $\m'$-primary.

    Now, $\varphi(e_{s_\ast, t_\ast}) =  e_{\varphi(s_\ast),\varphi(t_\ast)}$ directly
    follows from the definition of the Scheja-Storch element, since all
    involved equations are compatible with applying $\varphi$; see Definition
    \ref{def:scheja-storch}.
  \end{proof}
\end{lemma}

\begin{lemma}
  \label{lem:trace-sum}
  Let $L,L'$ be finite $k$-algebras. Let $A = L \times L'$
  be the product algebra. Then we have
  \[
    \Tr_{A/k}(a) = \Tr_{L/k}(\pr_L(a)) + \Tr_{L'/k}(\pr_{L'}(a))\in k
  \]
  for $a \in A$, where $\pr_L\colon A \to L$ and $\pr_{L'}\colon A\to L'$
  denote the projections to $L$ and $L'$, respectively.

  In particular,
  \[
    \Tr_{A/k}(\langle 1\rangle) = \Tr_{L/k}(\langle 1\rangle) +
    \Tr_{L'/k}(\langle 1\rangle)
  \]
  as quadratic forms over $k$, and $\Tr_{A/k}(\langle 1\rangle)$ is
  non-degenerate if and only if $\Tr_{L/k}(\langle 1\rangle)$ and
  $\Tr_{L'/k}(\langle 1\rangle)$ are.

  \begin{proof}
    Let $x_1, \dots x_r$ be a $k$-basis of $L$ and let $y_1, \dots, y_s$ be a
    $k$-basis of $L'$. Then $x_1, \dots, x_r, y_1, \dots, y_s$ is a $k$-basis of
    $A$.

    Since $x_iy_j = 0 \in A$, $x_ix_{i'}\in L\times \{0\} \subset A$, and
    $y_jy_{j'}\in \{0\}\times L'\subset A$, a simple computation shows
    \[
      \Tr_{A/k}(x_i) = \Tr_{L/k}(x_i) \quad\text{and}\quad \Tr_{A/k}(y_j) = \Tr_{L'/k}(y_j).
    \]
    Since $\pr_{L'}(x_j) = 0$ and $\pr_L(y_i) = 0$, this yields
    \begin{equation}
      \label{eq:trace-sum:base-sum}
      \Tr_{A/k}(a) = \Tr_{L/k}(\pr_L(a)) + \Tr_{L/k}(\pr_{L'}(a))
    \end{equation}
    for every $a \in \{x_1, \dots, x_r, y_1, \dots, y_s\}$. Since both sides are
    $k$-linear, \eqref{eq:trace-sum:base-sum} holds for all $a \in A$.

    Now, note that a straightforward computation also yields that
    \eqref{eq:trace-sum:base-sum} implies the in particular
    part of the statement since the matrix of $\Tr_{A/k}(\langle 1\rangle)$ with
    respect to the basis $x_1, \dots, x_r, y_1, \dots, y_s$ equals the direct
    sum of the matrices of $\Tr_{L/k}(\langle 1\rangle)$ and
    $\Tr_{L'/k}(\langle 1\rangle)$ with respect to $x_1, \dots, x_r$ and $y_1,
    \dots, y_s$, respectively.
  \end{proof}
\end{lemma}

\begin{corollary}
  \label{cor:separable-polynomial-trace}
  Let $f \in k[x]$ be a separable polynomial, i.e. assume that $f = p_1\cdots
  p_r$ for pairwise coprime irreducible polynomials $p_1, \dots, p_r \in
  k[x]$ (or equivalently assume that $f$ and $f'$ are coprime). Then
  $\Tr_{\frac{k[x]}{(f)}/k}(\langle 1\rangle)$ is non-degenerate and
  \[
    \Tr_{\frac{k[x]}{(f)}/k}(\langle 1\rangle) =
    \sum_{i=1}^r\Tr_{\frac{k[x]}{(p_i)}/k}(\langle 1\rangle) \in \GW(k).
  \]
  \begin{proof}
    By the Chinese remainder theorem, we have
    \[
      \frac{k[x]}{(f)} = \prod_{i=1}^r\frac{k[x]}{(p_i)}
    \]
    and hence, the claim follows from Lemma \ref{lem:trace-sum}.
  \end{proof}
\end{corollary}

\begin{rem}
  \label{rem:preimage-description}
  If $x\in \critical(\rho)$, we have the following local picture: let
  $t$ be a local generator of $L$ around $x$. So locally around $x$, we can
  write $s = s_xt^n$ for some $s_x \in \OO_{X,x}$. Then every local ring
  $\OO_{Y,y}$ with $y \in f^{-1}(x)$ is a factor of the ring
  $\OO_{X,x}[z]/(z^n-s_x)$. Indeed, this follows directly from the description
  of $Y$ as a fibre product of $X$ and $L$.
\end{rem}

\begin{proposition}
  \label{prop:etale-euler-class-change}
  Let $y \in f^{-1}(\critical(\rho))$ be a critical point and $x = f(y)$. Then
  \[
    \Tr_{k(y)/k(x)}(e_y(d(\rho\circ f))) = e_{x}(d\rho)\cdot
    \Tr_{k(y)/k(x)}(\langle 1\rangle) \in \GW(k(x)),
  \]
  where we consider $\langle 1\rangle$ as an element in $\GW(k(y))$.

  \begin{proof}
    By Corollary
    \ref{cor:riemann-hurwitz-ss}, we can use the language of Scheja-Storch
    forms to prove the statement.

    By Lemma
    \ref{lem:etale-scheja-storch-change}, the Scheja-Storch form on
    $\OO_{Y,y}$ is the extension of scalars of the Scheja-Storch
    form on $\OO_{X,x}$. Thus, by Lemma \ref{lem:trace-quadratic-linear}, we get
    \[
      \Tr_{k(y)/k(x)}(e_y(d(\rho\circ f))) = \Tr_{k(y)/k(x)}(e_x(d(\rho))\cdot
      \langle 1\rangle) = e_x(d\rho)\cdot \Tr_{k(y)/k(x)}(\langle 1\rangle)
      \in \GW(k(x)).\qedhere
    \]
  \end{proof}
\end{proposition}

\begin{corollary}
  \label{cor:etale-euler-class-change-sum}
  Let $x \in \critical(\rho)$ be a critical point, let $t$ be a local generator
  of $L$ around $x$, and write $s = s_xt^n$ with $s_x \in \OO_{X,x}$. Then
  \[
    \sum_{y \in f^{-1}(x)}\Tr_{k(y)/k(x)}(e_y(d(\rho\circ f))) =
    \begin{cases}
      e_{x}(d\rho)\cdot (\langle n \rangle + \frac{n-1}{2}\cdot H) & \text{if $n$ is odd},\\
      e_{x}(d\rho)\cdot (\langle n\rangle + \langle ns_x\rangle + \frac{n-2}{2}\cdot H)& \text{if
                                                                   $n$ is even.}
    \end{cases}
  \]
  in $\GW(k(x))$.

  \begin{proof}
    Note that $s_x \in \OO_{X,x}^\times$ since $\critical(\rho)\cap C = \emptyset$.
    Since $n$ is invertible in $k$, the polynomial $p = z^n-s_x$ has $n$ distinct
    roots in a splitting field of $p$. Thus, we can combine Proposition
    \ref{prop:etale-euler-class-change}, Corollary
    \ref{cor:separable-polynomial-trace}, and Lemma \ref{lem:cyclic-trace} to get the
    statement.
  \end{proof}
\end{corollary}

In two special cases of the local ring in Remark~\ref{rem:preimage-description},
we can also give a different description of the change of the Euler class.

\begin{corollary}
  \label{cor:etale-irreducible-contribution}
  Let $x \in \critical(\rho)$ be a critical point, $t$ a local
  generator of $L$ around $x$ and let $s_x \in \OO_{X,x}$ with $s = s_xt^n$.
  Suppose $z^n-s_x$ is irreducible in $k(x)[z]$. Then $f^{-1}(x)$ only contains a single point $y$ and we
  have
  \[
    \Tr_{k(y)/k(x)}(e_y(d(\rho\circ f))) = e_x(d\rho)\cdot
    \beta \in \GW(k(x))
  \]
  where
  \[
    \beta =
    \begin{cases}
      \langle n \rangle + \frac{n-1}{2}\cdot H & \text{if $n$ is odd,}\\
      \langle n\rangle + \langle ns_x\rangle + \frac{n-2}{2}\cdot H& \text{if
                                                                     $n$ is even}
    \end{cases}
  \]
  in $\GW(k(x))$.

  \begin{proof}
    Note that $\OO_{X,x}[z]/(z^n-s_x)$ is local and thus, by Remark
    \ref{rem:preimage-description}, the preimage contains exactly one point
    $y$. Furthermore, the
    local ring at $y$ is given by $\OO_{X,x}[z]/(z^n-s_x)$.

    Now, the statement is a special case of Corollary \ref{cor:etale-euler-class-change-sum}.
  \end{proof}
\end{corollary}

\begin{rem}
  Note that the quadratic form $\langle ns_x\rangle$ in Corollary
  \ref{cor:etale-irreducible-contribution} does not depend on the
  chosen trivialisation of $L$ if $n$ is even. Indeed, if $s_x'$ is obtained
  by choosing a different trivialisation, then $s_x = \lambda^ns_x'$ for some
  $\lambda \in \OO_{X,x}^\times$. Thus, we have $\langle ns_x \rangle =
  \langle n\lambda^ns_x'\rangle = \langle ns_x'\rangle$ because $n$ is even.
\end{rem}

\begin{corollary}
  \label{cor:etale-linear-contribution}
  Let $x \in \critical(\rho)$ be a critical point, $t$ a local
  generator of $L$ around $x$ and let $s_x \in \OO_{X,x}$ with $s = s_xt^n$.
  Suppose $z^n-s_x$ factors into $n$ different linear factors. Then $f^{-1}(x)$ contains
  exactly $n$ points and we have
  \[
    \sum_{y \in f^{-1}(x)}\Tr_{k(y)/k(x)}(e_y(d(\rho\circ f))) = n \cdot
    e_x(d\rho) \in \GW(k(x)).
  \]

  \begin{proof}
    Note that $\OO_{X,x}[z]/(z^n-s_x) \cong \OO_{X,x}^n$ since $z^n-s_x$ is
    separable and thus, by
    Remark \ref{rem:preimage-description}, the preimage contains exactly $n$
    points. Furthermore, the local ring of every $y \in f^{-1}(x)$ is
    isomorphic to $\OO_{X,x}$.

    Thus, we get $\Tr_{k(y)/k(x)}(e_y(d(\rho\circ f))) = e_x(d\rho)$ for every
    $y \in f^{-1}(x)$ because the trace is trivial as $k(y) = k(x)$. Summing
    over all points yields the statement.
  \end{proof}
\end{corollary}

\begin{rem}
  \label{rem:etale-algebraically-closed-relation}
  If $k$ is algebraically closed, all closed points are rational points. In
  particular, we are always in the situation  of 
  Corollary \ref{cor:etale-linear-contribution}. So, for $k = \C$,
  Corollary \ref{cor:etale-linear-contribution} recovers the
  classical relation between the Milnor numbers $\sum_{y \in f^{-1}(x)} \mu_y(\rho\circ f) = n \cdot
  \mu_x(\rho)$ by taking the rank.
\end{rem}

The following proposition was explained to the author by Jesse Pajwani.

\begin{proposition}
  \label{prop:real-etale-part-no-product}
  Let $k = \R$. Then there exists an finite étale torsor $Y \to X$ under
  $\Z/2\Z$ between smooth, projective surfaces over $\R$ such that there is no
  $\beta \in \GW(\R)$ with $\chi(Y/\R) = \beta \cdot \chi(X/\R)$.

  \begin{proof}
    Let $X$ be the real Enriques surface described in \cite[Appendix~C, p.~227, example with
    $E_\R^{(1)} = S$ and $E_\R^{(2)} = S$]{DIK2000RealEnriquesSurfaces}. By
    \cite[Statement~9.2.1]{DIK2000RealEnriquesSurfaces}, $X$ has two torsors under
    $\Z/2\Z$, both of which are K3-surfaces. Let
    $f\colon Y \to X$ be one of these torsors. Since $Y$ is a K3-surface, we have $\chi^\topo(Y(\C)) = 24$ and hence
    $\chi^\topo(X(\C)) = \frac 12\cdot \chi^\topo(Y(\C)) = 12$.

    Let $X':=
    f(Y(\R)) \subset X(\R)$. By our choice of $X$ and by
    \cite[Statement~9.2.1]{DIK2000RealEnriquesSurfaces}, $X(\R)$ is homeomorphic to the
    disjoint union of two $2$-spheres $S^2\amalg S^2$ and $X' \cong S^2$. Hence
    we have $\chi^\topo(X(\R)) = 4$ and $\chi^\topo(X') = 2$. Since $Y(\R)$ is a
    double cover of $X'$, we have $\chi^\topo(Y(\R)) = 4$.

    Putting this together, we have by Remark \ref{rem:euler-char-rank-signature}
    \[
      \rank\chi(X/k) = 12,\quad \rank \chi(Y/k)= 24,\quad \sgn\chi(X/k) =4,\quad
      \text{and}\quad \sgn\chi(Y/k) = 4,
    \]
    where $\sgn\colon \GW(\R) \to \Z$ denotes the signature homomorphism.
    Suppose there exists $\beta \in \GW(\R)$ such that $\chi(Y/k) =
    \beta\cdot\chi(X/k)$. Since rank and signature are both multiplicative, we must have
    \[
      \rank\beta = 2\quad \text{and}\quad \sgn\beta = 1,
    \]
    but such an element does not exist in $\GW(\R)$, since we always have
    $\rank\beta \equiv \sgn\beta\mod 2$.
  \end{proof}
\end{proposition}

\begin{rem}
  \label{rem:general-etale-part-no-product}
  By Proposition \ref{prop:etale-euler-class-change}, the change of the Euler
  classes in the étale contribution depends on the covering behaviour at the
  individual points. This can yield quite different forms as illustrated in
  Corollary~\ref{cor:etale-irreducible-contribution} and Corollary~\ref{cor:etale-linear-contribution}.
  Furthermore, in general, there is no
  element $\beta \in \GW(k)$ such that
  \[
    \sum_{y \in f^{-1}(\critical(\rho))}\Tr_{k(y)/k}(e_y(d(\rho\circ f))) = \beta \cdot
    \sum_{x \in \critical(\rho)}\Tr_{k(x)/k}(e_x(d\rho)) \in \GW(k)
  \]
  by Proposition \ref{prop:real-etale-part-no-product}.
\end{rem}

\subsection{The Branched Contribution}
\label{sec:odd-contribution}
In this section, we assume that $f\colon Y \to X$ is an $n$-fold geometrically
cyclic covering branched at $i\colon Z\hookrightarrow X$ with $n\ge 2$ invertible in $k$, and that there
is a morphism $\rho\colon X \to C$ to a smooth, projective curve possessing a
half-canonical line bundle such that $\critical(\rho)$ and $\critical(\rho\circ
i)$ are finite and disjoint.

\begin{theorem}
  \label{thm:odd-fixed-contribution}
  Let $x \in Z$ be a critical point.
  Then if $n$ is odd, we have
  \[
    e_x(d(\rho\circ f)) = \mu_x(\rho\circ i)\cdot
    \frac{n-1}{2}\cdot H \in \GW(k(x)),
  \]
  where $\mu_x(\rho\circ i)$ is the Milnor number from Remark
  \ref{rem:milnor-no-rank}.

  If $n$ is even, choose a local trivialisation of $L$ around
  $x$ and let $s_x \in \OO_{X,x}$ be the element corresponding with
  the section $s$ under the chosen trivialisation. Let $t$ be a
  normalised parameter of $C$ at $\rho(x)$ and write $d\rho(dt) =
  \alpha ds_x$ in $\Omega^1_{X,x}\otimes k(x)$ with $\alpha \in k(x)^\times$.
  Then we have
  \[
    e_x(d(\rho\circ f)) = e_x(\rho\circ i)\cdot
    (\langle n\alpha\rangle + \frac{n-2}{2}\cdot H) \in \GW(k(x)).
  \]

  \begin{proof}
    First note that we can define $\alpha$ as in the statement of the
    theorem because by our assumption on $\rho$, the point $x$ is a
    critical point of $\rho\circ i$, but not a critical point of $\rho$.

    By Corollary \ref{cor:riemann-hurwitz-ss}, we can use the language of
    Scheja-Storch forms to prove the statement. If $n$ is even, we
    pick the local trivialisation of $L$ from the statement. If $n$ is
    odd, choose a local trivialisation
    of $L$ and let $s_x \in \OO_{X,x}$ be the element corresponding with the
    section $s\colon X \to L^{\otimes n}$ under the chosen trivialisation of
    $L$ and let
    $t$ be a normalised parameter of $C$ at $\rho(x)$.

    Let $t_0 = s_x, t_1, \dots, t_r$ be a set of parameters of $X$ at $x$. By our assumption on $\rho$, we
    have $d\rho(dt) \ne 0$ in $\Omega_{X/k,x}\otimes k(x)$ and
    $d(\rho\circ i)(dt) = 0$ in $\Omega_{Z/k,x}\otimes k(x)$.

    Let $A = \OO_{X,x}$. Hence, we have $\OO_{Z,x} = A/(t_0)$ and $\OO_{Y,x} =
    A[z]/(z^n-t_0)$. Thus, $t_1, \dots, t_r$ is a set of parameters of $Z$ at
    $x$ and $z, t_1, \dots, t_r$ is a set of parameters of $Y$ at $x$.

    We can express $d\rho(dt) = \sum_{i=0}^r s_idt_i \in
    \Omega_{X/k,x}$ with $s_i \in A$. Note that $s_0 =
    \alpha$ in $k(x)$ if $n$ is even. This implies
    \[
      d(\rho\circ i)(dt) = \sum_{i=1}^r s_idt_i \in \Omega_{Z/k,x}
      \quad \text{and}\quad
      d(\rho\circ f)(dt) = ns_0z^{n-1}dz + \sum_{i=1}^r s_idt_i \in \Omega_{Y/k,x}
    \]
    since $dt_0 = nz^{n-1}dz\in \Omega_{Y/k,x}$. Because $d(\rho\circ i)(dt)$
    vanishes, we can write $s_i = \sum_{j=0}^r a_{ij}t_j \in A$ for $i = 1, \dots
    r$. And thus, we can also write
    \[
      s_i = \sum_{j=1}^r a_{ij}t_j \in A/(t_0)
      \quad \text{and}\quad
      s_i = z_iz + \sum_{j=1}^r a_{ij}t_j \in A[z]/(z^n-t_0)
    \]
    where $z_i = a_{i0}z^{n-1}$. We also define $z_0 = ns_0
    z^{n-2}$. Let $J_Z = \frac{A/(t_0)}{(s_1, \dots,
      s_r)}$ and $J_Y = \frac{A[z]/(z^n-t_0)}{(z_0z, s_1, \dots, s_r)}$. Hence,
    we have
    \[
      J_Y = \frac{k(x)[z]}{(z^{n-1})}\otimes_{k(x)} J_Z
    \]
    since $ns_0$ is a unit in $A$.
    Denote the projection of $A$ to $J_Z$ and $J_Y$ by $p_Z$ and $p_Y$,
    respectively.

    We are now turning to the computation of the Scheja-Storch bilinear forms:
    over $Z$, we get the Scheja-Storch generator
    $e_Z = p_Z(\det((a_{ij})_{1\le i,j\le r})) \in J_Z$. Let $\hat e_Z
    =\det((a_{ij})_{1\le i,j\le r}) \in A$.  Over $Y$, we obtain
    the matrix
    \[
      \begin{pmatrix}
        z_0 & 0& \dots & 0\\
        z_1 & a_{11} & \dots & a_{1r}\\
        \vdots & \vdots & & \vdots\\
        z_r & a_{r1} & \dots & a_{rr}\\
      \end{pmatrix}
    \]
    as our generator matrix.
    This matrix has determinant $z_0\cdot \hat e_Z$. And thus, the Scheja-Storch
    generator over $Y$ is given by $p_Y(z_0\cdot \hat e_Z)$, which
    corresponds with $z_0 \otimes e_Z$ under the above identification.

    Let $\Tr_Z\colon J_Z \to k(x)$ be a $k(x)$-linear map sending $e_Z$ to
    $1$, and let $\varphi\colon k(x)[z]/(z^{n-1})\to k(x)$ be the map given by
    sending $z^i$ to $0$ for $0 \le i < n-2$ and $z_0 = ns_0z^{n-2}$ to $1$. Hence, the map $\Tr_Y
    := \varphi\otimes \Tr_Z\colon J_Y \to k(x)$ has the property that
    $\Tr_Y(e_Y) = 1$.

    Let $B_Z$ be the bilinear form $(a,b) \mapsto \Tr_Z(ab)$ on $J_Z$, and
    $B_Y$ be the bilinear form $(a,b) \mapsto \Tr_Y(ab)$ on
    $J_Y$. Furthermore, denote the bilinear form $(a,b) \mapsto
    \varphi(ab)$ on $k(x)[z]/(z^{n-1})$ by $B_\varphi$. Then our trace
    computation yields $B_Y = B_\varphi \cdot B_Z$ as bilinear
    forms. By Corollary \ref{cor:riemann-hurwitz-ss}, $B_Y$ and
    $B_Z$ compute the Euler classes; hence
    \begin{equation}
      \label{eq:odd-fixed-contribution:form-relation}
      e_x(d(\rho\circ f)) = e_x(\rho\circ i)\cdot
      [B_\varphi] \in \GW(k(x)).
    \end{equation}

    If $n$ is even, the form $B_\varphi$ equals $\langle
    n\alpha\rangle + \frac{n-2}{2}\cdot H$ in $\GW(k(x))$. Indeed, the elements $z^i$
    for $i = 0, \dots, n-2$ form a basis of $k(x)[z]/(z^{n-1})$, and
    the pairs $(z^i, z^{n-2-i})$ for $i = 0, \dots, \frac{n-2}{2}-1$
    yield hyperbolic summands of $B_\varphi$ in $\GW(k(x))$. The
    remaining basis element $z^{(n-2)/2}$ pairs with itself to
    $B_\varphi(z^{(n-2)/2},z^{(n-2)/2}) = \frac{1}{ns_0} = \frac{1}{n\alpha} \in
    k(x)$. Since we
    have $B_\varphi(z^i,z^j) = 0$ for all $i,j \in \{0, \dots, n-2\}$
    not of the above form, we get the desired description of
    $B_\varphi$. Hence,
    \eqref{eq:odd-fixed-contribution:form-relation} yields the theorem
    for $n$ even.

    If $n$ is odd, a similar argument yields that $B_\varphi$ is hyperbolic because
    $n-1$ is even and thus equals $\dim_{k(x)}(k(x)[z]/(z^{n-1}))/2\cdot H =
    \frac{n-1}{2}\cdot H$. Hence,
    \eqref{eq:odd-fixed-contribution:form-relation} yields the theorem
    for $n$ odd.
  \end{proof}
\end{theorem}

\begin{rem}
  \label{rem:algebraically-closed-branched-contribution}
  If $k$ is algebraically closed, Theorem \ref{thm:odd-fixed-contribution}
  reduces to the simplified relation between the Milnor numbers
  \[
    \mu_x(\rho\circ f) = (n-1)\mu_x(\rho\circ i)
  \]
  by taking the rank. This recovers the classical relation for $k =
  \C$.
\end{rem}

\begin{corollary}
  \label{cor:odd-euler-char-general}
  Let $X$ be equidimensional of dimension $r$. If $n$ is odd, then we have
  \[
    \chi(Y/k) = (\langle n\rangle + \frac{n-1}{2}\cdot H)\cdot \chi(X/k)
     - \chi(Z/k)\cdot
                      \frac{n-1}{2}\cdot H
  \]
  in $\GW(k)$.

  If $n$ is even, choose an element $\alpha_x \in k(x)$ as in Theorem
  \ref{thm:odd-fixed-contribution} for each $x \in \critical(\rho\circ
  i)$ and an element $s_x \in k(x)$ as in Corollary
  \ref{cor:etale-euler-class-change-sum} for each $x \in \critical(\rho)$. Then we have in the notation of Theorem \ref{thm:riemann-hurwitz}
  \[
    \chi(Y/k) = (2\cdot \langle n\rangle + \frac{n-2}{2}H)\cdot
    \chi(X/k) + \beta_X - (\langle 1\rangle + \frac{n-2}{2}\cdot H)\cdot
    \chi(Z/k) - \beta_Z
  \]
  in $\GW(k)$ with
  \[
    \beta_X = (-1)^r\sum_{x \in \critical(\rho)}\Tr_{k(x)/k}(e_{x}(d\rho)\cdot(\langle
    ns_x\rangle - \langle n\rangle))\in \GW(k)
  \]
  and
  \[
    \beta_Z = (-1)^{r}\sum_{y \in
      \critical(\rho\circ i)}\Tr_{k(y)/k}
    (e_y(d(\rho\circ i)) \langle (n\alpha_x\rangle -
    \langle 1\rangle)) \in \GW(k).
  \]

  \begin{proof}
    By Theorem \ref{thm:riemann-hurwitz} and Proposition
    \ref{prop:critical-decomp}, we have
    \begin{align*}
      (-1)^r\chi(Y/k)
      &= \sum_{y \in \critical(\rho\circ f)}\pi_{Y\ast}i_{y\ast}e_y(d(\rho\circ f)) -
                        D(\rho\circ f)H\\
      &= \sum_{y \in f^{-1}(\critical(\rho))}\pi_{Y\ast}i_{y\ast}e_y(d(\rho\circ f)) + \sum_{y \in
        \critical(\rho\circ i)}\pi_{Y\ast}i_{y\ast}e_y(d(\rho\circ f)) -
                        D(\rho\circ f)H.
    \end{align*}
    Now, Corollary \ref{cor:riemann-hurwitz-ss} and Corollary
    \ref{cor:etale-euler-class-change-sum} yield
    \[
      \sum_{y \in f^{-1}(x)}\pi_{Y\ast}i_{y\ast}e_y(d(\rho\circ f)) = 
      \pi_{X\ast}i_{x\ast}(e_{x}(d\rho)\cdot \beta_x)
    \]
    for $x \in \critical(\rho)$ with
    \[
      \beta_x =
      \begin{cases}
        (\langle n \rangle + \frac{n-1}{2}\cdot H) & \text{if $n$ is odd},\\
        (\langle n\rangle + \langle ns_x\rangle + \frac{n-2}{2}\cdot H)& \text{if
                                                                   $n$ is even.}
      \end{cases}
    \]
    and hence
    \[
      (-1)^r\chi(Y/k) = \sum_{x \in \critical(\rho)}\pi_{X\ast}i_{x\ast}(e_{x}(d\rho)\cdot \beta_x) + \sum_{y \in
        \critical(\rho\circ i)}\pi_{Y\ast}i_{y\ast}e_y(d(\rho\circ f)) -
                        D(\rho\circ f)H
    \]
    We start with the case where $n$ is even.
    By Corollary \ref{cor:riemann-hurwitz-ss} and Theorem
    \ref{thm:odd-fixed-contribution}, we can rewrite the second summand and obtain
    \begin{align*}
      (-1)^r\chi(Y/k)
      &= \sum_{x \in \critical(\rho)}\pi_{X\ast}i_{x\ast}(e_{x}(d\rho)\cdot \beta_x) + \sum_{y \in
        \critical(\rho\circ i)}\pi_{Y\ast}i_{y\ast}e_y(d(\rho\circ f)) -
        D(\rho\circ f)H\\
      &= \sum_{x \in \critical(\rho)}\pi_{X\ast}i_{x\ast}(e_{x}(d\rho)\cdot \beta_x) \\
      &\quad+ \sum_{y \in \critical(\rho\circ i)}\Tr_{k(y)/k}\left(\mu_y(\rho\circ i)\cdot
        (\langle n\alpha_y\rangle + \frac{n-2}{2}\cdot H)\right)
       - D(\rho\circ f)H\\
      &= \sum_{x \in \critical(\rho)}\pi_{X\ast}i_{x\ast}(e_{x}(d\rho)\cdot \beta_x) \\
      &\quad+
        \sum_{y \in \critical(\rho\circ i)}\Tr_{k(y)/k}(e_y(d(\rho\circ
        i))(\langle n\alpha_y\rangle - \langle 1\rangle))\\
      &\quad +
        \sum_{y \in \critical(\rho\circ i)}\Tr_{k(y)/k}(e_y(d(\rho\circ i)))\cdot
        (\langle 1\rangle + \frac{n-2}{2}\cdot H)
       - D(\rho\circ f)H\\
      &= \sum_{x \in \critical(\rho)}\pi_{X\ast}i_{x\ast}(e_{x}(d\rho)\cdot \beta_x)\\
      &\quad+ \sum_{y \in \critical(\rho\circ i)}\Tr_{k(y)/k}(e_y(d(\rho\circ
        i))(\langle n\alpha_y\rangle - \langle 1\rangle))\\
      &\quad+
        ((-1)^{r-1}\chi(Z/k) + D(\rho\circ i)\cdot H)\cdot
        (\langle 1\rangle + \frac{n-2}{2}\cdot H)
       - D(\rho\circ f)H\\
    \end{align*}
    where we used Lemma \ref{lem:trace-quadratic-linear} for the third
    equality and that $Z$ has dimension $r-1$ by \cite[Zariski purity of the
    branch locus,
    \href{https://stacks.math.columbia.edu/tag/0BMB}{Tag~0BMB}]{stacks-project}
    for the last equality. We can also rearrange the summand from the étale part
    using Corollary \ref{cor:riemann-hurwitz-ss}, Corollary \ref{cor:etale-euler-class-change-sum}, and Lemma \ref{lem:trace-quadratic-linear}
    \begin{align*}
      &\sum_{x \in \critical(\rho)}\pi_{X\ast}i_{x\ast}(e_{x}(d\rho)\cdot
      \beta_x)\\
      =& \sum_{x \in \critical(\rho)}\Tr_{k(x)/k}(e_{x}(d\rho)\cdot
         (\langle n\rangle + \langle ns_x\rangle + \frac{n-2}{2}H))\\
      =& \sum_{x \in \critical(\rho)}\Tr_{k(x)/k}(e_{x}(d\rho)\cdot(\langle
         ns_x\rangle - \langle n\rangle)) + \Tr_{k(x)/k}(e_{x}(d\rho))\cdot(
         2\cdot\langle n\rangle + \frac{n-2}{2}H)\\
      =&\sum_{x \in \critical(\rho)}\Tr_{k(x)/k}(e_{x}(d\rho)\cdot(\langle
         ns_x\rangle - \langle n\rangle)) + ((-1)^r\chi(X/k) + D(\rho)\cdot H)\cdot(
         2\cdot \langle n\rangle + \frac{n-2}{2}H).
    \end{align*}
    If we plug this into the formula for $\chi(Y/k)$
    and note that $H\cdot \gamma = \rank\beta \cdot H$ for $\gamma
    \in \GW(k)$, we obtain
    \begin{align*}
      (-1)^r\chi(Y/k) &= \sum_{x \in \critical(\rho)}\Tr_{k(x)/k}(e_{x}(d\rho)\cdot(\langle
         ns_x\rangle - \langle n\rangle)) + (-1)^r\chi(X/k)\cdot(
         2\cdot \langle n\rangle + \frac{n-2}{2}H)\\
      &\quad+ \sum_{y \in \critical(\rho\circ i)}\Tr_{k(y)/k}(e_y(d(\rho\circ
        i))(\langle n\alpha_y\rangle - \langle 1\rangle))\\
                      &\quad + (-1)^{r-1} \chi(Z/k)\cdot
                        (\langle 1\rangle + \frac{n-1}{2}\cdot H) +(nD(\rho) + (n-1)D(\rho\circ i) - D(\rho\circ f))H
    \end{align*}
    By construction of $D(\rho)$, there exists some $t \in C$ such that
    $\rho^{-1}(t)$, $(\rho\circ i)^{-1}(t)$, and $(\rho\circ f)^{-1}(t)$ are
    all smooth. Hence, we have
    \begin{align*}
      & nD(\rho) + (n-1)D(\rho\circ i) - D(\rho\circ f)\\
      = & (-1)^{r-1}\frac 12 \chi^\topo(C) \cdot (n\cdot
          \chi^\topo(\rho^{-1}(t)) - (n-1)\chi^\topo((\rho\circ i)^{-1}(t)) -
          \chi^\topo((\rho\circ f)^{-1}(t)))\\
      = & 0
    \end{align*}
    by the Riemann-Hurwitz formula for the topological Euler
    characteristic. This yields the claim.

    If $n$ is odd, the proof is the same but simpler, since we then have
    \[
      \pi_{Y\ast}i_{y\ast}e_y(d(\rho\circ f)) = \Tr_{k(y)/k}\left(e_y(\rho\circ i)\cdot
        \frac{n-1}{2}\cdot H\right) = \Tr_{k(y)/k}(e_y(\rho\circ i))\cdot
        \frac{n-1}{2}\cdot H
    \]
    for $y \in \critical(\rho\circ i)$ and
    \[
      \pi_{X\ast}i_{x\ast}(e_{x}(d\rho)\cdot
      \beta_x) = \left(\langle n\rangle + \frac{n-1}{2}H\right) \cdot \pi_{X\ast}i_{x\ast}(e_{x}(d\rho))
    \]
    for $x \in \critical(\rho)$.
  \end{proof}
\end{corollary}

\begin{rem}
  \label{rem:recover-classical-formula}
  If $k = \C$, Corollary \ref{cor:odd-euler-char-general} recovers the classical formula
  \[
    \chi^\topo(Y(\C)) = n\chi^\topo(X(\C)) - (n-1)\chi^\topo(Z(\C))
  \]
  from topology by taking ranks.
\end{rem}

\begin{corollary}
  \label{cor:odd-euler-char-general-no-curve}
  If $n$ is odd, we can drop the assumption that there is a morphism from $X$ to a curve
  possessing a half-canonical line bundle in Corollary \ref{cor:odd-euler-char-general}.

  \begin{proof}
    We proceed by induction on the dimension $r$ of $X$. If $r = 0$ or $r =1$,
    we automaticllay have a morphism to a curve, because by Remark
    \ref{rem:map-construction}, we only need to blow up the empty subscheme to
    obtain a map. Hence, in this situation, we can directly apply Corollary
    \ref{cor:odd-euler-char-general}.
    
    Now, assume that $r \ge 2$. By Remark \ref{rem:map-construction}, we can find
    a smooth, closed subscheme $X' \subset X$ of codimension $2$ satisfying the
    following conditions:

    \begin{itemize}
    \item $X'\cap Z\subset Z$ is smooth and has codimension $2$.
    \item there exists a morphism $\Bl_{X'}X \to \PP^1$ such that the
      geometrically cyclic branched covering
      \[
        \begin{tikzcd}
          &\Bl_{f^{-1}(X')}Y\ar[d]\\
          \Bl_{X'\cap Z}Z\ar[r] & \Bl_{X'}X
        \end{tikzcd}
      \]
      satisfies the assumptions of Corollary \ref{cor:odd-euler-char-general}.
    \end{itemize}

    In particular $f^{-1}(X') \subset Y$ also has codimension $2$.

    Hence, Corollary \ref{cor:odd-euler-char-general} yields
    \[
      \chi(\Bl_{f^{-1}(X')}Y/k) = (\langle n\rangle + \frac{n-1}{2}\cdot H)\cdot \chi(\Bl_{X'}X/k)
      - \chi(\Bl_{X'\cap Z}Z/k)\cdot \frac{n-1}{2}\cdot H.
    \]
    Remark \ref{rem:euler-char-properties} and Example \ref{ex:euler-char-projective-space} yield
    \[
      \chi(Y/k) = \chi(\Bl_{f^{-1}(X')}Y/k) -\langle -1\rangle\cdot \chi(f^{-1}(X')/k)
    \]
    as well as
    \[
      \chi(\Bl_{X'}X/k) = \chi(X/k) +\langle -1 \rangle\cdot \chi(X'/k)
      \quad\text{and} \quad
      \chi(\Bl_{Z\cap X'}Z/k)=\chi(Z/k) +\langle -1 \rangle\cdot \chi(Z\cap X'/k).
    \]
    By the induction hypothesis for the $n$-fold geometrically cyclic branched covering
    \[
      \begin{tikzcd}
        &f^{-1}(X')\ar[d]\\
        X'\cap Z\ar[r] & X'
      \end{tikzcd}
    \]
    we also have
    \[
      \chi(f^{-1}(X')/k) = (\langle n\rangle + \frac{n-1}{2}\cdot H)\cdot \chi(X'/k)
      - \chi(Z\cap X'/k)\cdot \frac{n-1}{2}\cdot H
    \]
    If we plug all of these formulae into each other, we get the statement.
  \end{proof}
\end{corollary}

If all critical points in $\critical(\rho)$ have exactly $n$ preimages, we can
again reformulate Corollary \ref{cor:odd-euler-char-general}.

\begin{corollary}
  \label{cor:odd-euler-char-splitting}
  Let $X$ be equidimensional. Assume that $n$ is odd and that the polynomial $z^n - s_x \in k(x)[z]$ from
  Remark \ref{rem:preimage-description} factors into $n$ different linear factors for all $x
  \in \critical(\rho)$. Then we have
  \[
    \chi(Y/k) = n\cdot  \chi(X/k) -
    \chi(Z/k)\cdot \frac{n-1}{2}\cdot H\in \GW(k).
  \]

  \begin{proof}
    In the proof of Corollary \ref{cor:odd-euler-char-general}, we can replace
    $\beta_x$ by $n\cdot \langle 1\rangle$ by Corollary
    \ref{cor:etale-linear-contribution}. Then the proof of Corollary
    \ref{cor:odd-euler-char-general} yields the statement.
  \end{proof}
\end{corollary}

\begin{rem}
  As in Remark~\ref{rem:general-etale-part-no-product}, the change of the Euler class in
  Theorem~\ref{thm:odd-fixed-contribution} depends on the critical point when
  $n$ is even. This gives little hope that, in general, there is an
  element $\beta\in \GW(k)$ such that
  \[
    \sum_{x \in f^{-1}(\critical(\rho\circ i))}\Tr_{k(x)/k}(e_x(d(\rho\circ f))) = \beta \cdot
    \sum_{x \in \critical(\rho\circ i)}\Tr_{k(x)/k}(e_x(d\rho)) \in \GW(k).
  \]
\end{rem}

\section{Application: Quadratic Euler Characteristic of Double Coverings of
  $\PP^2$}
\label{sec:application-p2}
In this section, we consider double coverings of $\PP^2$. That is, we consider
branched coverings of the form
\[
  \begin{tikzcd}
    & X \ar[r]\ar[d, "\varphi"] & \OO(n)\ar[d, "m_2"]\\
    C \ar[r, "i"] & \PP^2 \ar[r, "s"] & \OO(2n)
  \end{tikzcd}
\]
where $n\ge 1$ is invertible in $k$ and we assume that $X$ has a rational point
not on $C$. If $n = 3$, then $X$ is a K3-surface by \cite[Chapter 1, Example
1.3~(iv)]{Huybrechts2016K3}.

\begin{lemma}
  \label{lem:k3:rational-point-splitting}
  If $y \in X(k)$ is a rational point not on $C$, then $\varphi^{-1}(\varphi(y)) \subset X$
  consists of two rational points.

  \begin{proof}
    Let $s_y \in \OO_{\PP^2,\varphi(y)}$ be an element corresponding with $s$ under a
    local trivialisation of $\OO(2n)$ around $\varphi(y)$. Then $\varphi$ induces the following commutative
    diagram:
    \[
      \begin{tikzcd}
        \OO_{\PP^2,\varphi(y)} \ar[r, hook] \ar[d, "\varphi(y)"] &
        \OO_{\PP^2,\varphi(y)}[t]/(t^2-s_y)\ar[d]\\
        k \ar[r, hook] & k[t]/(t^2-s_y).
      \end{tikzcd}
    \]
    where we note that $\varphi(y)$ is also a rational point and the map $\varphi(y)$ in the
    diagram denotes taking the residue field.
    We can identify $\OO_{X,\varphi^{-1}(\varphi(y))} = \OO_{\PP^2,\varphi(y)}[t]/(t^2-s_y)$ and under this
    identification, the rational point $y$ corresponds with a retract of the
    inclusion $k \hookrightarrow k[t]/(t^2-s_y)$. Since $s_y$ is invertible in
    $\OO_{\PP^2,\varphi(y)}$, we thus get that
    $k[t]/(t^2-s_y) \cong k\times k$ as a ring for degree reasons, and thus
    $\varphi^{-1}(\varphi(y)) \subset X$ consists of two rational points.
  \end{proof}
\end{lemma}

Take such a point $y \in X(k)$ and consider its image $x:= \varphi(y) \in
\PP^2$. By choosing coordinates on $\PP^2$ correctly, we can assume that $x
= [0:0:1] \in \PP^2$. In these coordinates, we can represent the section $s$ by a homogeneous polynomial $F \in k[X_0,X_1,X_2]$ of degree $2n$. Consider the rational map $\rho\colon \PP^2
-\to \PP^1; [X_0:X_1:X_2] \mapsto [X_0:X_1]$. This map is defined on
$\PP^2\setminus \{x\}$ and it defines a well-defined map $\rho\circ i\colon C
\to \PP^1$.

The rational map $\rho$ induces a morphism $\tilde\rho \colon \Bl_{x}\PP^2
\to \PP^1$ and
\[
  \begin{tikzcd}
    & \Bl_{\varphi^{-1}(x)}X \ar[r]\ar[d, "\tilde \varphi"] & \OO(n)\ar[d, "m_2"]\\
    C \ar[r, "i"] & \Bl_{x}\PP^2 \ar[r, "s"] & \OO(2n)
  \end{tikzcd}
\]
is again a geometrically cyclic branched covering. In the following, we shall compute the
quadratic Euler characteristic of $\Bl_{\varphi^{-1}(x)}X$.

\begin{rem}
  \label{rem:k3:blow-up-reduction}
  We can deduce $\chi(X/k)$ from $\chi(\Bl_{\varphi^{-1}(x)}/k)$ using the blow-up
  formula in Remark \ref{rem:euler-char-properties}: the preimage $\varphi^{-1}(x)$
  consists of two rational points, and hence $\chi(\varphi^{-1}(x)/k) =
  2\langle 1\rangle$. In particular, we get
  \begin{align*}
    \chi(X/k)
    &= \chi(\Bl_{\varphi^{-1}(x)}X/k) + (\langle 1\rangle -
      \chi(\PP^1/k))\cdot \chi(\varphi^{-1}(x)/k)\\
    &= \chi(\Bl_{\varphi^{-1}(x)}X/k) -
    2\langle -1\rangle \in \GW(k).
  \end{align*}
\end{rem}

\begin{lemma}
  \label{lem:k3:critical-P2}
  The induced map $\tilde\rho \colon \Bl_{x}\PP^2 \to \PP^1$ has no
  critical points.

  \begin{proof}
    Over $D_{\PP^2}(X_0)$ and $D_{\PP^2}(X_1)$, the projection $\pi\colon \Bl_{x}\PP^2 \to
    \PP^2$ is an isomorphism. Furthermore for $i = 0,1$, the morphism
    $\rho\colon D_{\PP^2}(X_i) \to \PP^1$ induces an inclusion $d\rho\colon
    \rho^\ast\Omega^1_{\PP^1}|_{D_{\PP^1}(X_i)} \to \Omega^1_{\PP^2}|_{D_{\PP^2}(X_i)}$. Hence,
    $\rho$ has no critical points on $D_{\PP^2}(X_0)$ and $D_{\PP^2}(X_1)$. Since these open
    subsets cover $\PP^2\setminus\{x\}$, the morphism $\tilde\rho\colon \Bl_{x}\PP^2 \to \PP^1$ has no critical
    points outside of the exceptional divisor.

    Using the isomorphism $D_{\PP^2}(X_2) \cong \A^2$, we can model $\pi^{-1}(D(X_2))$,
    which includes the entire exceptional divisor, as $\Bl_{0}\A^2$. Under this
    identification, $\rho$ becomes the morphism induced by the rational morphism
    $\rho\colon \A^2 -\to \PP^1, (x,y) \mapsto [x:y]$.

    By \cite[Chapter II, Example 7.12.1]{HartshorneAG}, we can model $\Bl_0\A^2$
    as $\Proj(k[x,y][T_x,T_y]/(xT_y-yT_x))$ where $x$ and $y$ have degree $0$,
    and $T_x$ and $T_y$ have degree $1$. In this model, the morphism $\tilde\rho\colon
    \Bl_0\A^2\to \PP^1$ is given by $[T_x:T_y]\mapsto [T_x:T_y]$.

    In particular, $\tilde\rho^{-1}(D_{\PP^1}(X_0))= \{T_x \ne 0\}$ and
    $\tilde\rho^{-1}(D_{\PP^1}(X_0)) = \{T_y \ne 0\}$. Thus, over $D_{\PP^2}(X_0)$, we get that $d\tilde\rho
    \colon \tilde\rho^\ast\Omega^1_{\PP^1/k}|_{D_{\PP^2}(X_0)} \to \Omega^1_{\Bl_0\A^2/k}|_{T_x\ne 0}$ is
    given by $d(X_1/X_0) \mapsto d(T_y/T_x)$, and this is an inclusion. Hence $d\tilde\rho$
    vanishes nowhere on $D_{\PP^1}(X_0)$. Since we can repeat the same argument for
    $D_{\PP^1}(X_1)$, we get the claim.
  \end{proof}
\end{lemma}

\begin{corollary}
  \label{cor:k3:cover-critical-points}
  We have $\critical(\tilde\rho\circ \tilde\varphi) = \critical(\tilde\rho\circ
  i) = \critical(\rho\circ i)$.

  \begin{proof}
    Combine Lemma \ref{lem:k3:critical-P2} and Proposition
    \ref{prop:critical-decomp} and notice that $\rho\circ i = \tilde\rho\circ i$.
  \end{proof}
\end{corollary}

\begin{lemma}
  \label{lem:curve-critical-points}
  We have
  \[
    \critical(\rho\circ i) = \left\{\frac{\partial F}{\partial X_2} =
      0\right\}\subset C
  \]
  and this set is a finite collection of closed points in $C$.

  \begin{proof}
    We can cover $C$ with $D_{\PP^2}(X_i) \subset \PP^2$ for $i = 0,1$. Let $j =
    1-i$. Then we have on $D_{\PP^2}(X_i)$ that
    \[
      \Omega^1_{C/k}|_{D_{\PP^2}(X_i)} \cong \frac{\OO_{C\cap D_{\PP^2}(X_i)/k}dx_j \oplus
        \OO_{C\cap D_{\PP^2}(X_i)/k}dx_2}{\frac{\partial f}{\partial x_j}dx_j +\frac{\partial f}{\partial x_2}dx_2},
    \]
    where $x_j = X_j/X_i$ and $x_2 = X_2/X_i$ are the coordinate of $D(X_i)
    \cong \Spec k[x_j,x_2]$ and $f$ is the dehomogenisation of $F$ with respect
    to $X_i$.

    Since we have $\rho^{-1}(D_{\PP^1}(X_\nu)) = D_{\PP^2}(X_\nu)$ for $\nu = 1,2$, we get over
    $D_{\PP^1}(X_i)\subset \PP^1$ that $d\rho(dx_j) = dx_j$. Since $C$ is smooth, we
    have that for every closed point at least one of
    $\frac{\partial f}{\partial x_j}$ and $\frac{\partial f}{\partial x_2}$ has
    to be invertible at this point. In particular, $dx_j \in
    \Omega^1_{C/k}|_{D_{\PP^2}(X_i)}$ vanishes if and only if $\frac{\partial
      f}{\partial x_j}$ is invertible and $\frac{\partial f}{\partial x_2}$ is
    non-invertible. Now $\frac{\partial f}{\partial x_2}$ being non-invertible
    implies that $\frac{\partial f}{\partial x_j}$ is invertible because $C$ is
    smooth. Furthermore, since $\frac{\partial f}{\partial
      x_2}$ is the dehomogenisation with respect to $X_i$ of $\frac{\partial
      F}{\partial X_2}$, we get the claim.

    Since $2n$ is invertible in $k$, the map $\rho\circ i$ is separated and
    generically unramified. Hence, $\rho\circ i$ only has finitely many critical points.
  \end{proof}
\end{lemma}

\begin{lemma}
  \label{lem:k3:curve-local-euler-class}
  Let $y \in \critical(\tilde \rho\circ i)$ and choose $j \in \{0,1\}$ with $y \in
  D(X_{1-j})$. Let $f \in k[x_{j},x_2]$ be the dehomogenisation of $F$ with
  respect to
  $X_{1-j}$. Let $t_y \in \OO_{C,y}$ be a local parameter and write
  $\frac{\partial f}{\partial x_2} = ut_y^m$ in $\OO_{C,y}$ with $m \ge 1$ and
  $u \in \OO_{C,y}^\times$. Suppose that $m$ is invertible in $k$. Then we can write $\frac{\partial^2 f}{\partial
    x_2^2} = m \alpha t_y^{m-1}$ for some $\alpha \in \OO_{C,y}^\times$ and we
  have
  \[
    e_y(d(\rho\circ i)) =
    \begin{cases}
      \langle -\alpha \cdot  \frac{\partial f}{\partial x_{j}}\rangle +
      \frac{m-1}{2}\cdot H,& \text{for $m$ odd}\\
      \frac{m}{2} \cdot H, & \text{for $m$ even}
    \end{cases}
  \]
  in $\GW(k(y))$.

  In particular, if $\frac{\partial f}{\partial x_2}$ is a local
  parameter of $\OO_{C,y}$, then $\frac{\partial^2 f}{\partial
    x_2^2} \in \OO_{C,y}^\times$ and we have
  \[
    e_y(d(\rho\circ i)) =
      \langle -\frac{\partial^2 f}{\partial
        x_2^2} \cdot  \frac{\partial f}{\partial x_{j}}\rangle \in \GW(k(y)).
  \]

  \begin{proof}
    We use Corollary \ref{cor:riemann-hurwitz-ss} to compute $e_y(\rho\circ
    i)$. As in the proof of Lemma~\ref{lem:curve-critical-points}, we have
    \[
      \Omega^1_{C/k}|_{D_{\PP^2}(X_{1-j})} \cong \frac{\OO_{C\cap D_{\PP^2}(X_{1-j})/k}dx_j \oplus
        \OO_{C\cap D_{\PP^2}(X_{1-j})/k}dx_2}{\frac{\partial f}{\partial x_j}dx_j +\frac{\partial f}{\partial x_2}dx_2}
    \]
    and we need to compute the local Euler class with respect to the section
    $dx_j$. Note that $df = \frac{\partial f}{\partial x_j}dx_j +\frac{\partial
      f}{\partial x_2}dx_2$ and similarly $d\frac{\partial f}{\partial x_2} =
    \frac{\partial^2 f}{\partial x_j\partial x_2}dx_j +\frac{\partial^2
      f}{\partial x_2^2}dx_2$.

    There are $\alpha', \beta \in \OO_{\PP^2,y}^\times$
    such that
    \[
      dt_y = \frac{\beta}{u}dx_j + \frac{\alpha'}{u}dx_2
    \]
    in $\Omega^1_{\PP^2,y}$.
    Since $\frac{\partial f}{\partial x_2} = ut_y^m$, we also have
    \[
      umt_y^{m-1}dt_y = d\frac{\partial f}{\partial x_2} =
      \frac{\partial^2 f}{\partial x_j\partial x_2}dx_j +\frac{\partial^2
        f}{\partial^2 x_2}dx_2
    \]
    in $\Omega^1_{C,y}$ and hence there exists an $r \in \OO_{\PP^2,y}$ such
    that
    \[
      umt_y^{m-1}dt_y + rdf = \frac{\partial^2 f}{\partial x_j\partial x_2}dx_j +\frac{\partial^2
        f}{\partial x_2^2}dx_2
    \]
    in $\Omega^1_{\PP^2,y}$.
    Expanding $df$ and using that $dx_j,dx_2$ is an $\OO_{\PP^2,y}$-basis of
    $\Omega^1_{\PP^2,y}$, we get
    \[
      \frac{\partial^2f}{\partial x_2^2} = mt_y^{m-1}\alpha' + r\frac{\partial
        f}{\partial x_2} = mt_y^{m-1}\alpha' + rut_y^m =
      m\underset{=:\tilde\alpha}{\underbrace{(\alpha' + rut_y)t_y^{m-1}}} =
      m\tilde\alpha t_y^{m-1}
    \]
    in $\OO_{\PP^2,y}$.
    
    Let
    \[
      A =
      \begin{pmatrix}
        \frac{\partial f}{\partial x_j} & \beta/u\\
        \frac{\partial f}{\partial x_2} & \alpha'/u
      \end{pmatrix}.
    \]
    Then $A$ is the change of basis matrix from the $\OO_{\PP^2,y}$-basis $(df, dt_y)$ to the
    basis $(dx_j, dx_2)$ for $\Omega^1_{\PP^2,y}$ and hence
    invertible. In particular because $\frac{\partial f}{\partial x_2}$ vanishes
    in $k(y)$, we get $\alpha' \in \OO_{\PP^2,y}^\times$ and thus
    $\tilde \alpha \in \OO_{\PP^2,y}^\times$. Hence we can indeed write $\frac{\partial^2 f}{\partial
      x_2^2} = m \tilde \alpha t_y^{m-1}$ in $\OO_{C,y}$ with $\tilde \alpha \in \OO_{\PP^2,y}^\times$ and since this decomposition is
    unique, we have $\tilde \alpha = \alpha$ in $\OO_{C,y}$ and $\alpha'$ and $\alpha$ have the same residue class in $k(y)$.

    Furthermore, we can
    write
    \[
      dx_j = (\det A)^{-1} \cdot (\frac{\alpha'}{u}\cdot df - \frac{\partial f}{\partial x_2}dt_y)
    \]
    in $\Omega^1_{\PP^2,y}$.
    Thus, we get the Scheja-Storch generator $-(\det A)^{-1}ut_y^{m-1} \in J = \OO_{C,y}/(ut_y^m)$. The linear form $J \to
    k(y)$ sending
    $t_y^i$ to $0$ for $i < m-1$ and $t^{m-1}_y$ to $-\frac{\det A}{u} =
    -\frac{\alpha'}{u^2} \cdot  \frac{\partial f}{\partial x_{j}} = -\frac{\alpha}{u^2} \cdot  \frac{\partial f}{\partial x_{j}} \in k(y)$
    induces the Scheja-Storch form. If $m$ is even, the $t_y^i$ for $i = 0,
    \dots, m-1$ pair to hyperbolic forms such that we get the result in that
    case. If $m$ is odd, $t_y^{{m-1}/2}$ pairs with itself to $-\frac{\alpha}{u^2} \cdot
    \frac{\partial f}{\partial x_{j}}$.  Since $\langle -\frac{\alpha}{u^2} \cdot
    \frac{\partial f}{\partial x_{j}}\rangle = \langle -\alpha \cdot
    \frac{\partial f}{\partial x_{j}}\rangle \in \GW(k(y))$, this yields the claim.
  \end{proof}
\end{lemma}

\begin{proposition}
  \label{prop:k3:local-euler-k3}
  Let $y \in \critical(\rho\circ i)$ and choose $j \in \{0,1\}$ with $y \in
  D(X_{1-j})$. Let $f \in k[x_{j},x_2]$ be the dehomogenisation of $F$ with
  respect to
  $X_{1-j}$. Let $t_y \in \OO_{C,y}$ be a local parameter and write
  $\frac{\partial f}{\partial x_2} = ut_y^m$ in $\OO_{C,y}$ with $m \ge 1$ and
  $u \in \OO_{C,y}^\times$. Suppose that $m$ is invertible in $k$. Then we can write $\frac{\partial^2 f}{\partial
    x_2^2} = m \alpha t_y^{m-1}$ for some $\alpha \in \OO_{C,y}^\times$ and we
  have
  \[
    e_y(d(\tilde \rho\circ \tilde \varphi)) =
    \begin{cases}
      \langle -2\alpha \rangle +
      \frac{m-1}{2}\cdot H& \text{for $m$ odd}\\
      \frac{m}{2} \cdot H & \text{for $m$ even}
    \end{cases}
  \]
  in $\GW(k(y))$.

  \begin{proof}
    In $\Omega^1_{\PP^2,y}\otimes k(y)$, we have $df = \frac{\partial
      f}{\partial x_j}dx_j$ because $\frac{\partial f}{\partial x_2} = 0$ in
    $k(y)$. In particular, we have
    \[
      dx_j = \left(\frac{\partial f}{\partial x_j}\right)^{-1}df.
    \]
    The result now follows by combining Lemma
    \ref{lem:k3:curve-local-euler-class} and Theorem \ref{thm:odd-fixed-contribution}.
  \end{proof}
\end{proposition}

\begin{corollary}
  \label{cor:k3:local-euler-k3-nice-local-gen}
  In the situation of Proposition \ref{prop:k3:local-euler-k3}, assume that
  $\frac{\partial f}{\partial x_2}$ is a local generator of $\OO_{C,y}$ (or
  equivalently assume that $m = 1$). Then
  \[
    e_y(d(\tilde \rho\circ \tilde\varphi)) = \langle -2\cdot \frac{\partial^2 f}{\partial
      x_2^2}\rangle \in \GW(k(y)).
  \]

  \begin{proof}
    Note that under the assumption, we can take $t_y = \frac{\partial f}{\partial
      x_2}$ and $u = 1$ in Proposition \ref{prop:k3:local-euler-k3}.
  \end{proof}
\end{corollary}

\begin{rem}
  Even in the nice case where all critical points satisfy the additional
  assumption of Corollary \ref{cor:k3:local-euler-k3-nice-local-gen}, and after
  taking traces and adjusting something, the traces are not related to the traces
  of the critical points of the curve
  $\frac{\partial F}{\partial X_2}$. Indeed, by the proof of Lemma
  \ref{lem:curve-critical-points}, the critical points of the induced map to
  $\PP^1$ are the points where $\frac{\partial^2 F}{\partial^2 X_2^2}$ and
  $\frac{\partial F}{\partial X_2}$ are both zero, but in Corollary
  \ref{cor:k3:local-euler-k3-nice-local-gen} we are looking at the points where
  $F$ and $\frac{\partial F}{\partial X_2}$ are zero. These two sets don't agree
  in general.
\end{rem}

\begin{theorem}
  \label{thm:k3:euler-char}
  For every $y \in \critical(\rho\circ i)$, choose $\alpha_y \in \OO_{C,y}$ and
  $m_y \ge 1$ as in
  Lemma \ref{lem:k3:curve-local-euler-class}. Assume that $m_y$ is
  invertible in $k$ for all $y\in \critical(\rho\circ i)$.
  Define
  \[
    \beta = \sum_{\substack{y \in \critical(\rho\circ i)\\m_y\text{ odd}}}
    \Tr_{k(y)/k}(\langle -2\alpha_y\rangle) \in \GW(k).
  \]
  Then $\rank\beta$ is even and
  \[
    \chi(X/k) = 2\langle1\rangle + \beta + ((2n-1)(n-1) + 1 -\frac 12 \rank
    \beta)\cdot H \in \GW(k).
  \]
  \begin{proof}
    By Remark \ref{rem:k3:blow-up-reduction}
    \[
      \chi(X/k) = \chi(\Bl_{\varphi^{-1}(x)}X/k) -
      2\langle -1 \rangle = 2\langle 1\rangle + \chi(\Bl_{\varphi^{-1}(x)}X/k) - 2H.
    \]
    We can compute using Remark \ref{rem:euler-char-properties}
    \[
      \chi(\Bl_x\PP^2/k) = \chi(\PP^2/k) - \chi(\Spec k/k) + \chi(\PP^1/k) = 2H
    \]
    and by the genus degree formula, $C$ has genus $(2n-1)(n-1)$. Hence, the classical covering formula from topology yields
    \begin{align*}
      \rank \chi(\Bl_{\varphi^{-1}(x)}X/k)
      &= 2\rank\chi(\Bl_{x}\PP^2/k) - \rank\chi(C/k)
      = 2\cdot (4 - (1-
      (2n-1)(n-1)))\\
      &= 2\cdot (3 + (2n-1)(n-1)).
    \end{align*}
    By Corollary \ref{cor:riemann-hurwitz-ss}, we have for some $r \in \Z$
    \begin{align*}
      \chi(\Bl_{\varphi^{-1}(x)}X/k)
      &= \sum_{y\in \critical(\tilde \rho\circ \tilde\varphi)}\Tr_{k(y)/k}e_y(\rho\circ f) + rH\\
      &= \sum_{y\in \critical(\rho\circ i)}\Tr_{k(y)/k}e_y(\rho\circ i) + rH\\
      &= \sum_{\substack{y\in \critical(\rho\circ i)\\m_y\text{
      odd}}}\Tr_{k(y)/k}(\langle -2\alpha_y\rangle + \frac{m_y-1}{2}H) +
      \sum_{\substack{y\in \critical(\rho\circ i)\\m_y\text{
      even}}}\Tr_{k(y)/k}(\frac{m_y}{2}H) + rH\\
      &= \sum_{\substack{y\in \critical(\rho\circ i)\\m_y\text{
      odd}}}\Tr_{k(y)/k}(\langle -2\alpha_y\rangle) +
      \sum_{\substack{y\in \critical(\rho\circ i)\\m_y\text{
      odd}}}\Tr_{k(y)/k}(\frac{m_y-1}{2}\langle 1\rangle)H\\
      &\quad +
      \sum_{\substack{y\in \critical(\rho\circ i)\\m_y\text{
      even}}}\Tr_{k(y)/k}(\frac{m_y}{2}\langle 1\rangle)H + rH\\
      &= \beta +
      (\sum_{\substack{y\in \critical(\rho\circ i)\\m_y\text{
      odd}}}\Tr_{k(y)/k}(\frac{m_y-1}{2}\langle 1\rangle) +
      \sum_{\substack{y\critical(\in \rho\circ i)\\m_y\text{
      even}}}\Tr_{k(y)/k}(\frac{m_y}{2}\langle 1\rangle) + r)H\\
    \end{align*}
    Here, we used Corollary \ref{cor:k3:cover-critical-points} for the second
    equation, Proposition \ref{prop:k3:local-euler-k3} for the third equation,
    and Lemma \ref{lem:trace-quadratic-linear} for the forth equation. In
    particular, we get
    \[
      \chi(\Bl_{\varphi^{-1}(x)}X/k) = \beta + r'H
    \]
    for some $r' \in \Z$. Since $\rank\chi(\Bl_{\varphi^{-1}(x)}X/k)$ is even,
    this implies that $\rank\beta$ is even. By comparing ranks, we get
    \[
      r' = (2n-1)(n-1) + 3 - \frac 12\rank \beta.
    \]
    Combining all formulae, we get the claim.
  \end{proof}
\end{theorem}

\begin{proposition}
  \label{prop:k3:fermat-example}
  Consider the surface $X$ defined as the geometrically cyclic branched double
  covering of $\PP^2$ induced by the section $F = aX_0^{2n} + bX_1^{2n} +
  cX_2^{2n}$ of $\OO(2n)$ with $a,b,c \in k^\times$ and assume that $2n$ and $2n-1$ are
  invertible in $k$. Then
  \[
    \chi(X/k) = \langle 2\rangle + \langle 2ab\rangle + ((2n-1)(n-1)+1)H \GW(k).
  \]

  \begin{proof}
    Note that the covering and the statement do not change if we divide $F$ by
    $c$. Hence, we may assume that $c = 1$.

    Thus, $X$ has two rational points outside of $C$ over the rational point $x
    := [0:0:1]
    \in \PP^2$. In particular, we can use the results in this section directly to
    compute $\chi(X/k)$. We have
    \[
      \frac{\partial F}{\partial X_2} = 2nX_2^{2n-1} \quad \text{and}\quad
      \frac{\partial^2 F}{\partial X_2^2} = 2n\cdot (2n-1)\cdot X_2^{2n-2}.
    \]
    Hence, the critical points of $C = V(F)$ are all points with $X_2 = 0$, these are
    the points in the vanishing set $\{X_2 = 0, aX_0^{2n}+bX_1^{2n} = 0\}$. All
    points in this vanishing set lie in the intersection $D_{\PP^2}(X_0) \cap D_{\PP^2}(X_1)$ and
    hence, we can work in $D_{\PP^2}(X_0)$. There, the equation $aX_0^{2n}+bX_1^{2n} = 0$
    dehomogenises to
    \[
      a+bx_1^{2n} = 0 \Leftrightarrow \frac{a}{b} + x_1^{2n} = 0.
    \]
    Hence, the zero locus consists of at most $2n$ distinct closed points. For
    each closed point $y \in \{X_2 = 0, aX_0^{2n}+bX_1^{2n} = 0\}$, we can pick
    the local parameter $t_y = x_2$. This yields $u_y = 2n$, $m_y = 2n-1$, and $\alpha_y
    = 2n$ in Lemma \ref{lem:k3:curve-local-euler-class}. In particular, all
    $m_y$ are odd and, in Theorem \ref{thm:k3:euler-char}, we get using Lemma \ref{lem:trace-quadratic-linear}
    \[
      \beta = \sum_{y\in \{X_2 = 0, aX_0^{2n}+bX_1^{2n} =
        0\}}\Tr_{k(y)/k}(\langle -4n\rangle) = \langle -4n\rangle \cdot \sum_{y\in \{X_2 = 0, aX_0^{2n}+bX_1^{2n} =
        0\}}\Tr_{k(y)/k}(\langle 1\rangle).
    \]
    Now, note that by our analysis of the critical points, we have
    \[
      \prod_{y\in \{X_2 = 0, aX_0^{2n}+bX_1^{2n} =
        0\}} k(y) \cong \frac{k[x_1]}{(\frac ab + x_1^{2n})}
    \]
    since $\frac ab + x_1^{2n}$ is separable. Thus, Corollary
    \ref{cor:separable-polynomial-trace} and Lemma \ref{lem:cyclic-trace} yield
    \[
      \beta = \langle -4n\rangle \cdot (\langle 2n\rangle + \langle -2n\frac
      ab\rangle) +(n-1)H = \langle -2\rangle + \langle 2ab\rangle + (n-1)H.
    \]
    Now, Theorem \ref{thm:k3:euler-char} together with the relation $2\cdot
    \langle 1\rangle = 2\cdot \langle 2\rangle$ yields the claim.
  \end{proof}
\end{proposition}

\begin{example}
  \label{ex:k3:fermat-example:k3-case}
  For $n = 3$, we get
  \[
    \chi(X/k) = \langle 2\rangle + \langle 2ab\rangle + 11H \GW(k)
  \]
  in Proposition \ref{prop:k3:fermat-example}.
  This is a K3 surface.
\end{example}

\begin{example}
  \label{ex:k3:final-example}
  For $a = b = c= 1$ in Proposition \ref{prop:k3:fermat-example}, we can also
  give a different description of $\chi(X/k)$.

  We have
  \[
    \prod_{1\le d|4n}\Phi_d(x_1) = x_1^{4n}-1 = (x_1^{2n}+1)(x_1^{2n}-1) =
    (x_1^{2n}+1)\cdot \prod_{1\le d|2n}\Phi_d(x_1)
  \]
  where $\Phi_d$ is the $d$-th cyclotomic polynomial, whose degree we denote by $\phi(d)$. Hence
  \[
    (x_1^{2n}+1) = \prod_{\substack{1\le d|4n\\ d\nmid 2n}}\Phi_{d}(x_1) =
    \prod_{\substack{1\le d|n\\ 2d\nmid n}}\Phi_{4d}(x_1).
  \]
  Note that $x_1^{2n}+1$ is the polinomial whose zeroes correspond to critical
  points in the proof of Proposition \ref{prop:k3:fermat-example}.

  Since $4n$ is invertible in $k$, the polynomials $\Phi_{4d}$ are either irreducible
  in $k$ or they factor into $\phi(4d)$ distinct linear factors.
  Let $d\ge 1$ divide $n$. If $\Phi_{4d}$ is irreducible, then this factor corresponds with a
  unique closed point $x_{4d} \in C$. If $\Phi_{4d}$ factors into $\phi(4d)$ linear
  factors, $\Phi_{4d}$ corresponds with $\phi(4d)$ rational points. Hence, we
  can write $\beta$ in the proof of Proposition \ref{prop:k3:fermat-example} as
  \[
    \beta := \langle -4n\rangle\cdot\left(\sum_{{\substack{1\le d|n; 2d\nmid n\\ \Phi_{4d}
            \text{ irreducible}}}} \Tr_{k[\zeta_{4d}]/k}(\langle 1\rangle) +
      \sum_{{\substack{1\le d|n; 2d\nmid n\\ \Phi_{4d}
            \text{ factors}}}} \phi(4d)\cdot \langle 1\rangle\right) \in \GW(k).
  \]
  Now, Theorem \ref{thm:k3:euler-char} yields that $\rank\beta$ is even and
  \[
    \chi(X/k) = 2\langle 1\rangle + \beta + ((2n-1)(n-1) + 1 -\frac 12\rank\beta)H \in \GW(k).
  \]
  Also, note that $k$ contains a square root of $-1$ if any of the $\Phi_{4d}$
  factors; and hence, we have $2\langle 1\rangle = H$ in this case.

  For $k = \Q$ and $n = 3$, $X$ is a K3-surface, and we can explicitly recover
  the result of Example \ref{ex:k3:fermat-example:k3-case}. In order to compute
  $\chi(X/k)$, we need to compute $\Tr_{\Q[i]/\Q}(\langle 1\rangle)$ and
  $\Tr_{\Q[\zeta_{12}]/\Q}(\langle 1\rangle)$. A straightforward computation
  shows that $\Tr_{\Q[i]/\Q}(\langle 1\rangle) = H$. In order to compute
  $\Tr_{\Q[\zeta_{12}]/\Q}(\langle 1\rangle)$, note that $\Q[\zeta_{12}] =
  \Q[i,\zeta_3]$. Now, one can compute the field traces
  \[
    \Tr_{\Q[\zeta_{12}]}(1) = 4, \quad \Tr_{\Q[\zeta_{12}]}(\zeta_3) = -2, \quad
    \Tr_{\Q[\zeta_{12}]}(i) = 0, \quad \text{and}\quad \Tr_{\Q[\zeta_{12}]}(i\zeta_3) = 0
  \]
  using the identity $\zeta_3^2 =
  -1-\zeta_3$.
  Since $\{1,\zeta_3, i, i\zeta_3\}$ forms a $\Q$-basis of $\Q[\zeta_{12}]$, we
  can use these traces to compute $\Tr_{\Q[\zeta_{12}]/\Q}(\langle 1\rangle)$ as
  the quadratic form corresponding with the matrix
  \[
    \begin{pmatrix}
      4 & -2 & 0 & 0\\
      -2 & -2 & 0 & 0\\
      0 & 0 & -4& 2\\
      0 & 0 & 2 & 2
    \end{pmatrix},
  \]
  which corresponds with $2H \in \GW(\Q)$. Hence, we get $\beta = 3H$ and
  thus
  \[
    \chi(X/\Q) = 2\langle 1\rangle + 3H + ((6-1)(3-1) + 1 - 3)H = 2\langle 1\rangle + 11H \in \GW(\Q).
  \]
\end{example}

\begin{rem}
  Proposition \ref{prop:k3:fermat-example} shows that the condition of $m$ being
  invertible in
  Proposition \ref{prop:k3:local-euler-k3} is not always satisfied: if $k$ has odd
  characteristic $p > 0$, the section induced by $X_0^{p+1}+
  X_1^{p+1}+X_2^{p+1}$ does not satisfy that $m$ is invertible since $m = p$ in
  this case.
\end{rem}

\appendix
\section{Appendix: Lefschetz Pencils and Maps to Curves}
\label{sec:lefschetz}
In this section, we shall see how
one can construct a morphism on a blow-up of $X$ along a smooth, codimension
two subscheme $L$, such that $L$ intersects $Z$ transversely, using the
existence of Lefschetz pencils. Our discussion of Lefschetz pencils is based
on \cite[Exposé XVII]{SGA7-2}.

Recall the \emph{Veronese embedding} of $\PP^n$ of degree $d$, that is, the closed immersion
\[
i_{n,d}:\PP^n_k\to \PP_k^{N_{n,d}}
\]
defined by $H^0(\PP^n, \OO_{\PP^n}(d))$; here $N_{n,d}=\dim_kH^0(\PP^n, \OO_{\PP^n}(d))-1=
\binom{n+d}{n}-1$.

For a finite type $k$-scheme $Y$ of pure dimension $n\ge1$ over an
algebraically closed field $k$, we follow \cite[Exposé XVII, (1.1)]{SGA7-2} and call
a closed point $y\in Y$ an \emph{ordinary double point} if the completion
$\widehat{\OO}_{Y,y}$ of the local ring $\OO_{Y,y}$ is isomorphic to a
quotient of the power series ring over $k$ of the form $k[[T_1,\ldots,
T_{n+1}]]/(Q(T))$, with $Q(T)\in  k[[T_1,\ldots, T_{n+1}]]$, such that
\begin{itemize}
\item if we write $Q(T)$ as a sum of its homogeneous terms, we have $Q(T)=\sum_{i\ge
    2}Q_i(T)$, and
\item the closed subscheme of $\PP^n_k$ defined by $Q_2(T)$ is
  smooth over $k$.
\end{itemize}
After a change of variables, $Q_2(T)$ will be of the form
\[
  Q_2(T)=\begin{cases} \sum_{i=1}^{[\frac{n+1}{2}]}T_{2i-1}T_{2i}&\text{ for $n$ odd,}\\
    \sum_{i=1}^{[\frac{n}{2}]}T_{2i-1}T_{2i}+T_{n+1}^2&\text{ for $n\ge2$ even.}
  \end{cases}
\]

\begin{definition}
  Let $k$ be an algebraically closed field of characteristic not two. Let
  $i\colon X\hookrightarrow \PP^n_k$ be a closed immersion of a smooth, proper
  $k$-scheme.  Let $L\subset \PP^n_k$ be a codimension two linear subspace. We
  say that \emph{$L$ defines a Lefschetz pencil on $X$} if
  \begin{enumerate}
  \item $L$ intersects $X$ transversely
  \item Let $D=\{H\supset L\mid H\text{ a hyperplane in }\PP^n_k\}$. Then for
    all but finitely many $H\in D$, $H$ intersects $X$ transversely.
  \item If $H\in D$ does not intersect $X$ transversely, then there is a
    single singular point $x\in H\cap X$ and $x$ is an ordinary double point
    on $H\cap X$.
  \end{enumerate}
  The family of hyperplane sections $\{H\cap X\mid H\in D\}$ is called the
  \emph{Lefschetz pencil on $X$ defined by $L$}.
\end{definition}

If $k$ is not algebraically closed, we say that a codimension two linear
subspace $L\subset \PP^n_k$ defines a Lefschetz pencil for $i\colon
X\hookrightarrow \PP^n_k$ if this is the case after base-extension to the
algebraic closure of $k$.

\begin{rem}
  If $L\subset \PP^n$ defines a Lefschetz pencil for $i\colon
  X\hookrightarrow \PP^n$, then $H\cap X$ is smooth in an neighbourhood of
  $L\cap X$ for each hyperplane $H\supset L$. In other words, even if $H\cap
  X$ has a singular point $x$, then $x$ is not in $L\cap X$. Indeed, if $H'$
  is any other hyperplane containing $L$, then $H\cap H'=L$, so $(H\cap
  H')\cap X=H'\cap (H\cap X)$ is a  transverse intersection, which implies
  that  $H\cap X$ must be smooth in a neighbourhood of $L\cap X$.
\end{rem}

\begin{definition}
  Let $i\colon X\hookrightarrow \PP^n_k$ be a closed immersion of a smooth proper
  $k$-scheme $X$. We call $i$ a \emph{Lefschetz embedding} if the set of $L\in
  \Gr(n-2,\PP^n)$ such that $L$ defines a Lefschetz pencil for $i$ is an open
  dense subset of $\Gr(n-2,\PP^n)$.
\end{definition}

\begin{proposition}[{\cite[Exposé~XVII, Th\'eor\`eme~2.5]{SGA7-2}}]
  \label{prop:LefschetzPencil}
  Let $X$ be a smooth, projective scheme over $k$ of dimension $r$, with a
  given projective embedding $i\colon X\hookrightarrow \PP^n_k$. 
  \begin{enumerate}
  \item\label{prop:lef1} A general hyperplane in $\PP^n$ intersects $X$ transversely.
  \item\label{prop:lef2} For all $d\ge1$ if $k$ has characteristic zero, or
    $d\ge2$ if $k$ has characteristic $\neq2$, or $d\ge 3$ in general, the
    closed immersion $i_{n,d}\circ i\colon X\hookrightarrow \PP^{N_{n,d}}_k$
    is a Lefschetz embedding.
  \item\label{prop:lef3} Let $i_Z\colon Z\hookrightarrow X$ be a smooth closed
    subscheme of $X$ of dimension $r'<r$. Then for all $d\gg0$, the closed
    immersions $i_{n,d}\circ i$ and $i_{n,d}\circ i\circ i_Z$ are Lefschetz
    embeddings. Moreover, there is a dense open subset $U\subset
    \Gr(N_{n,d}-2, \PP^{N_{n,d}})$ such that each $L\in U$ defines a Lefschetz
    pencil for  $i_{n,d}\circ i$ and $i_{n,d}\circ i\circ i_Z$, and in
    addition, for all hyperplanes $H\supset L$ with singular intersection
    $H\cap X$, the singular point of $H\cap X$ is contained in $X\setminus Z$,
    and $H\cap Z$ is smooth over $k$.
  \end{enumerate}
\end{proposition}

\begin{proof}
  We may take $k$ to be algebraically closed. The statements
    \eqref{prop:lef1} and \eqref{prop:lef2} are \cite[Exposé~XVII,
    Th\'eor\`eme~2.5]{SGA7-2}. For \eqref{prop:lef3}, given an embedding
    $i:X\hookrightarrow \PP^n_k$, we have the dual projective space
    $\check{\PP}^n$ of hyperplanes in $\PP^n$, and the dual variety
    $\check{X}\subset \check{\PP}^n$ of hyperplanes that are tangent to $X$ at
    some point, that is, whose intersection with $X$ is singular. By
    \cite[Exposé~XVII, Proposition~3.1.4]{SGA7-2}, $\check{X}$ is irreducible
    of dimension $r-1$, and $i$ is a Lefschetz embedding if and only if the
    subset of $H\in \check{X}(\bar{k})$ such that $H\cap X$ has a single point
    $x$ of non-transverse intersection, with $x\in H\cap X$ an ordinary double
    point, defines a dense open subset of $\check{X}$. Moreover, this subset
    is always open (but may be empty).

    Applying \eqref{prop:lef1} and \eqref{prop:lef2}  for both $i$ and $i\circ
    i_Z$, we see that all the conditions of \eqref{prop:lef3} are satisfied,
    except possibly the condition that the singularity in $H\cap X$ is not in
    $Z$. We will show that the set of hyperplanes $H\in \check{X}$, such that
    the singular locus of $H\cap X$ is disjoint from $Z$, is a dense open
    subset of $\check{X}$. As this subset is clearly open, we need only show
    it is non-empty.

    Choose a closed point $x\in X\setminus Z$, let $\m_x\subset \OO_{X,x}$ be
    the maximal ideal and consider the coherent sheaves
    $\F_d:=\m_x^2\otimes_{\OO_X}\OO_X(d)$. Consider the blow-up $\pi\colon\tilde{X}\to X$ of $X$ at
    $x$ with exceptional divisor $E$. By
    Lemma~\ref{lem:VanishingDirectIm} below,
    $\F_d=\pi_*(\pi^\ast\OO_X(d)\otimes_{\OO_{\tilde{X}}}\OO_{\tilde{X}}(-2E))$
    and the canonical map $H^0(X, \F_d)\to H^0(\tilde{X},
    \pi^\ast\OO_X(d)\otimes_{\OO_{\tilde{X}}}\OO_{\tilde{X}}(-2E))$ is an
    isomorphism, for all $d$. Moreover, there is an integer $d_0\ge0$ such
    that for all $d\ge d_0$,
    $\pi^\ast\OO_X(d)\otimes_{\OO_{\tilde{X}}}\OO_{\tilde{X}}(-2E)$ is a very
    ample invertible sheaf on  $\tilde{X}$. Finally, increasing $d_0$ if
    necessary, we may assume that $\F_d$ is generated by global sections and
    the restriction map $H^0(X,\OO(d))\to H^0(X, \OO(d)\otimes \OO_X/\m_x^2)$
    is surjective.

    The isomorphism $H^0(X, \F_d)\cong H^0(\tilde{X},
    \pi^\ast\OO_X(d)\otimes_{\OO_{\tilde{X}}}\OO_{\tilde{X}}(-2E))$ can be
    interpreted geometrically as follows. Let $s$ be a global section of
    $\F_d$. Via the inclusion $\F_d\subset \OO_X(d)$, we may consider $s$ as a
    global section of $\OO_X(d)$. The fact that $s$ comes from $\F_d$ implies
    that the divisor $V(s)$ of $s$ has a singular point at $x$. The fact that
    the restriction map $H^0(X,\OO(d))\to H^0(X, \OO(d)\otimes \OO_X/\m_x^2)$
    is surjective says that the inclusion $\PP(H^0(X, \F_d))\hookrightarrow
    \PP(H^0(X,\OO_X(d))$ identifies $\PP(H^0(X, \F_d))$ with the linear
    subsystem of $\PP(H^0(X,\OO_X(d))$ consisting of all divisors that have a
    singular point at $x$. The fact that $\F_d$ is generated by global
    sections  says that, after trivializing $\OO_X(d)$ in a neighbourhood of
    $x$, the induced restriction map $H^0(X, \F_d)\to \m_x^2/\m_x^3$ is
    surjective.  In terms of the divisors $V(s)$, this says that for all $s$
    in a dense, open subscheme of $H^0(X, \F_d)$, the divisor $V(s)$ has an
    ordinary double point at $x$.  If $s$ maps to $s'\in
    H^0(\tilde{X}(\pi^\ast\OO_X(d)\otimes_{\OO_{\tilde{X}}}\OO_{\tilde{X}}(-2E))$
    under the isomorphism $H^0(X, \F_d)\cong
    H^0(\tilde{X}(\pi^\ast\OO_X(d)\otimes_{\OO_{\tilde{X}}}\OO_{\tilde{X}}(-2E))$,
    we have  $V(s)=\pi_\ast(V(s'))$, so 
    there is a dense open subscheme $U'_x$ of $H^0(\tilde{X},
    \pi^\ast\OO_X(d)\otimes_{\OO_{\tilde{X}}}\OO_{\tilde{X}}(-2E))$ such that,
    for all $s'\in U'_x$, the pushforward $\pi_\ast(V(s'))$ of the divisor
    $V(s')$ of $s'$ has an ordinary double point at $x$. On the other hand,
    since $ \pi^\ast\OO_X(d)\otimes_{\OO_{\tilde{X}}}\OO_{\tilde{X}}(-2E)$ is
    very ample, it follows from \eqref{prop:lef1} that for a general $s'\in
    H^0(\tilde{X},
    \pi^\ast\OO_X(d)\otimes_{\OO_{\tilde{X}}}\OO_{\tilde{X}}(-2E))$, the
    divisor $V(s')$ is smooth on $\tilde{X}$. Translating back to $H^0(X,
    \F_d)$, since $\pi\colon \tilde{X}\to X$ is an isomorphism away from $x$,
    we see that, after shrinking $U_x$ if necessary, for each $s\in U_x\subset
    H^0(X, \F_d)\subset H^0(X, \OO_X(d))$ the divisor $V(s)$ is smooth on
    $X\setminus\{x\}$ and has an ordinary double point at $x$. Similarly,
    since $x$ is not in $Z$, we may identify $Z$ with the closed subscheme
    $\pi^{-1}(Z)\subset \tilde{X}$. Applying (\ref{prop:lef1}) again, we see that for a
    general $s'\in H^0(\tilde{X},
    \pi^\ast\OO_X(d)\otimes_{\OO_{\tilde{X}}}\OO_{\tilde{X}}(-2E))$,
    $V(s')\cap Z$ is a smooth codimension one closed subscheme of $Z$.

    In particular, in the dual variety $\check{X}$ for the embedding
    $i_{n,d}\circ i$, there is point $p\in \check{X}$ corresponding with a
    hyperplane $H$ such that $H\cap X$ has the single singular point $x$, the
    point $x$ is an ordinary double point on $H\cap X$, and $x$ is not in
    $Z$. This shows that the open subset $\check{X}^0$ of such points of
    $\check{X}$ is non-empty, hence, a dense open subset, as claimed. In
    particular, the closed subset $\overline{\check{X}}:=\check{X}\setminus
    \check{X}^0$ is a proper closed subset of $\check{X}$, hence, has
    codimension at least two in $\PP^{N_{n,d}}$. Moreover, $p$ is not in the
    dual variety $\check{Z}$ for the embedding  $i_{n,d}\circ i\circ i_Z$, so
    the irreducible subscheme
    $\check{Z}$ does not contain $\check{X}$. Thus, replacing $\check{X}^0$
    with $\check{X}^0\setminus \check{Z}$ and changing notation, we may assume
    that for each $p\in \check{X}^0$, the corresponding hyperplane $H$ has
    smooth intersection $H\cap Z$ with $Z$.

    A codimension two linear subspace $L\subset \PP^{N_{n,d}}$ corresponds
    with a line $\ell$ in the dual projective space $\check{\PP}^{N_{n,d}}$,
    and by the above, there is a dense open subscheme $V\subset \Gr(1,
    \check{\PP}^{N_{n,d}})$ such that each  $\ell\in V$  will  define a
    Lefschetz pencil for both $i_{n,d}\circ i$ and $i_{n,d}\circ i\circ i_Z$,
    and in addition, $\ell\cap
    \overline{\check{X}}=\emptyset$. In other words, for each hyperplane $H$ in
    the pencil defined by $L$, and  having a singular intersection $H\cap X$, the
    singular point is not in $Z$, and $H\cap Z$ is smooth. Taking $U=\check{V}\subset  \Gr(N_{n,d}-2,
    \PP^{N_{n,d}})$  proves \eqref{prop:lef3}.
\end{proof}

\begin{rem}
  If $k$ is an infinite field, one can always find a Lefschetz pencil satisfying 
  Proposition~\ref{prop:LefschetzPencil} \eqref{prop:lef1}, \eqref{prop:lef2},
  \eqref{prop:lef3}
  and defined over $k$: such Lefschetz pencils correspond with $k$-points in a
  dense open subset $U$ of $\Gr(1, \check{\PP}^{N_{n,d}})$, and as $\Gr(1,
  \check{\PP}^{N_{n,d}})$ contains a dense open subset isomorphic to an affine
  space, $U$ will have infinitely many $k$-points for every infinite field
  $k$. However, if $k$ is a finite field, the arguments above do not show that
  $U(k)$ is non-empty.
\end{rem}

We conclude this appendix on Lefschetz pencils with the statement and proof
of the lemma used in the proof of Proposition~\ref{prop:LefschetzPencil} and
an explanation of how Lefschetz pencils can be used to satisfy the assumption
in Section \ref{sec:coverings-def}.

\begin{lemma}\label{lem:VanishingDirectIm}
  Let $W\subset Y$ be a smooth, closed subscheme of a smooth, separated,
  finite-type $k$-scheme $Y$ and let $\pi\colon\tilde{Y}\to Y$ be the blow-up of $Y$
  along $W$, with exceptional divisor $i\colon E\hookrightarrow \tilde{Y}$. Let
  $\I_W\subset\OO_Y$ denote the ideal sheaf of $W$ and
  $\OO_{\tilde{Y}}(-E)=\I_{\tilde{E}}\subset\OO_{\tilde{Y}}$ the ideal sheaf
  of $E$, an invertible sheaf on $\tilde{Y}$.
  \begin{enumerate}
  \item\label{lem:VanishingDirectIm1} For all $r\ge 0$,  we have
    $\pi_\ast\OO_{\tilde{Y}}(-rE)=\I_W^r$ and
    $R^i\pi_\ast\OO_{\tilde{Y}}(-rE)=0$ for $i>0$.
  \item\label{lem:VanishingDirectIm2} For each invertible sheaf $\mathcal{L}$
    on $Y$, the map $\pi_\ast$ induces an isomorphism $H^i(Y,
    \I_W^r\otimes\mathcal{L})\cong H^i(\tilde{Y},
    \OO_{\tilde{Y}}(-rE)\otimes\pi^\ast\mathcal{L})$.
  \item\label{lem:VanishingDirectIm3} For all $r\ge1$, the sheaf
    $\OO_{\tilde{Y}}(-rE)$ is $\pi$-relatively very ample.
  \item\label{lem:VanishingDirectIm4} Let $\mathcal{L}$ be a very ample
    invertible sheaf on $Y$ and fix $r\ge1$. Then for all $d\gg0$, the sheaf
    $\OO_{\tilde{Y}}(-rE)\otimes\pi^\ast\mathcal{L}^{\otimes d}$ is very ample
    on $\tilde{Y}$.
  \end{enumerate}
\end{lemma}

\begin{proof}
  \eqref{lem:VanishingDirectIm1} is \cite[Lemma 4.3.16]{Laz04PosI} and
  \eqref{lem:VanishingDirectIm2} follows from the projection formula
  \[
    R^i\pi_\ast(\OO_{\tilde{Y}}(-rE)\otimes_{\OO_{\tilde{Y}}}\pi^\ast\mathcal{L})\cong 
    R^i\pi_\ast(\OO_{\tilde{Y}}(-rE))\otimes_{\OO_{Y}}\mathcal{L},
  \]
  \eqref{lem:VanishingDirectIm1}, and the Leray spectral sequence.

  For 
  \eqref{lem:VanishingDirectIm3}, it follows from \eqref{lem:VanishingDirectIm1} that 
  \[
    \pi_\ast\Sym^\ast\OO_{\tilde{Y}}(-rE)=\pi_\ast(\oplus_{i\ge0}\OO_{\tilde{Y}}(-riE))=\oplus_{i\ge0}\I_W^{ri}.
  \]
  Since  $\tilde{Y}=\Proj_{\OO_Y}\oplus_{i\ge0}\I_W^{i}$ by definition, this
  settles the case of $r=1$ and for general $r$, note that
  $\Proj_{\OO_Y}\oplus_{i\ge0}\I_W^{ir}$ is just the image of
  $\Proj_{\OO_Y}\oplus_{i\ge0}\I_W^{i}$ under a relative Segre embedding of
  degree $r$, which takes care of the general case.

  Finally, we prove \eqref{lem:VanishingDirectIm4}. Let $t_0,\ldots, t_m$ be
  a basis for $H^0(Y, \mathcal{L})$ and let $U_i=Y\setminus V(t_i)$. Since
  $\mathcal{L}$ is very ample $U_0,\ldots, U_m$ is an affine open cover of
  $Y$; let $V_i=\pi^{-1}(U_i)$. By the equivalence of (3) and (5) for $d=1$ in
  \cite[\href{https://stacks.math.columbia.edu/tag/01VT}{Tag~01VT}]{stacks-project}, there is an $n\ge1$ such that,  for
  each $i$, there are  $s_0^{(i)},\ldots, s_n^{(i)}\in H^0(V_i,
  \OO_{\tilde{Y}}(-rE))$  defining a closed immersion $f_i\colon V_i\to
  U_i\times_k\PP^n$ over $U_i$. But then there is an integer $d_1\ge0$ such
  that $s_j^{(i)}t_i^{d_1}$ extends (uniquely) to
  $H^0(\tilde{Y}, \OO_{\tilde{Y}}(-rE)\otimes\pi^\ast\mathcal{L}^{\otimes
    d_1})$, and thus, the tuple of sections
  $(\ldots,s_j^{(i)}t_i^{d_1},\ldots)_{i=0,\ldots, m, j=0,\ldots, n}$ defines a
  closed immersion $f\colon \tilde{Y}\to Y\times \PP^N$, with
  $N=(m+1)(n+1)-1$. Since $\mathcal{L}$ is very ample, the tuple $(t_0,\ldots,
  t_m)$ defines a closed immersion $t\colon Y\to \PP^m_k$, giving us the
  closed immersion $(f,t\circ\pi):\tilde{Y}\to
  \PP^N_k\times_k\PP^m_k$. Composing with the Segre embedding $i_{N,m}\colon
  \PP^N_k\times_k\PP^m_k\to \PP^M$, with $M=(N+1)(m+1)-1$, we have the closed
  immersion $i_{N,m}\circ(f,t\circ\pi)\colon \tilde{Y}\to\PP^M_k$ defined by the tuple of sections $(\ldots,
  s_j^{(i)}t_i^{d_1}t_\ell,\ldots)_{i,j,\ell}$,  which shows that
  $\OO_{\tilde{Y}}(-rE)\otimes\pi^\ast\mathcal{L}^{\otimes d_1+1}$ is very
  ample. Since $\mathcal{L}$ is generated by global sections, we find that
  $\OO_{\tilde{Y}}(-rE)\otimes\pi^\ast\mathcal{L}^{\otimes d}$ is very ample for
  all $d\ge d_0:=d_1+1$.
\end{proof}

\begin{rem}
  \label{rem:map-construction}
  Let $f\colon Y\to X$ be an $n$-fold geometrically cyclic covering branched at
  $Z \subset X$. Let $L$ be the defining line bundle and let $s\colon X \to
  L^{\otimes n}$ be the section defining $Z$.
  By picking a closed immersion $X \to \PP^N$ and a Lefschetz pencil $L_0
  \subset \PP^{N}$ as in
  Proposition~\ref{prop:LefschetzPencil}~(\ref{prop:lef3}), we get a morphism
  $\rho\colon \Bl_{L_0\cap X}X \to \PP^1$ from the blow-up $\pi\colon
  \Bl_{L_0\cap X}X\to X$ of $X$ along
  $L_0\cap X$ to $\PP^1$. The closed immersion
  $Z\hookrightarrow X$ induces a closed immersion $\Bl_{Z\cap L_0}Z \hookrightarrow
  \Bl_{X\cap L_0}X$. Since $\Bl_{Z\cap L_0}Z$ is the preimage of $Z$ under
  $\pi$, we find that $\Bl_{Z\cap L_0}Z$ is the vanishing locus of the
  pullback section $\pi^\ast s\colon \Bl_{X\cap L_0}X \to \pi^\ast
  L$. Furthermore, the square
  \[
    \begin{tikzcd}
      \Bl_{f^{-1}(X\cap L_0)}Y \ar[r]\ar[d] & \pi^\ast L\ar[d, "m_n"]\\
      \Bl_{X\cap L_0}X \ar[r, "\pi^\ast s"] & \pi^\ast L^{\otimes n}
    \end{tikzcd}
  \]
  is a pullback square since blowing up commutes with flat base-change by
  \cite[\href{https://stacks.math.columbia.edu/tag/0805}{Tag~0805}]{stacks-project}.
  In particular $\Bl_{f^{-1}(X\cap L_0)}Y \to \Bl_{X\cap L_0}X$ is an $n$-fold
  geometrically cyclic covering branched at $\Bl_{Z\cap L_0}Z$. Moreover,
  note that by assumption on $L_0$, the map $\rho$ only has finitely many
  critical points, all of which lie outside of $\Bl_{Z\cap L_0}Z$ and the
  restriction of $\rho$ to $\Bl_{Z\cap L_0}Z$ also only has finitely many
  critical points.
\end{rem}

\printbibliography[heading=bibintoc]

~\\
	
\noindent Louisa F. Br\"oring \\
Technische Universität Darmstadt \\
Schlossgartenstraße 7,
64289 Darmstadt, Germany \\
E-Mail: \href{mailto:broering@mathematik.tu-darmstadt.de}{broering@mathematik.tu-darmstadt.de}\\ \\ 

\noindent Keywords: motivic homotopy theory, refined enumerative geometry, other fields\\ 
Mathematics Subject Classification: 14G27, 14N10, 14F42
\end{document}